\documentclass[10pt]{article}

\usepackage{graphicx}
\usepackage{amssymb}
\usepackage{amsmath}
\usepackage{amsthm}
\usepackage{verbatim}
\usepackage{url}
\usepackage{color}
\usepackage{epstopdf}

\newcounter{listacnt}\renewcommand{\thelistacnt}{\alph{listacnt}}

\newcommand{\bydef}{\,\stackrel{\mbox{\tiny\textnormal{\raisebox{0ex}[0ex][0ex]{def}}}}{=}\,}
\newcommand{\overa}{\bar{a}}
\newcommand{\bx}{\bar{x}}
\newcommand{\ellnu}{\ell_{\nu}^1}
\newcommand{\overL}{\bar{L}}
\newcommand{\overx}{\bar{x}}
\newcommand{\overp}{\bar{\psi}}
\newcommand{\R}{\mathbb{R}}
\newcommand{\hv}{\widehat{v}}
\newcommand{\Q}{\mathcal{Q}}

\newtheorem{theorem}{Theorem}
\newtheorem{definition}[theorem]{Definition}

\newtheorem{lemma}[theorem]{Lemma}
\newtheorem{proposition}[theorem]{Proposition}

\newtheorem{remark}[theorem]{Remark}

\numberwithin{equation}{section}

\begin{document}
\title{Continuation of homoclinic orbits in the suspension bridge equation: a computer-assisted proof}%
\author{Jan Bouwe van den Berg \footnote{VU University Amsterdam, Department of Mathematics, De Boelelaan 1081, 1081 HV Amsterdam, The Netherlands. janbouwe@few.vu.nl}
\and 
Maxime Breden \footnote{CMLA, ENS Cachan, CNRS, Universit\'e Paris-Saclay, 94235 Cachan, France and Universit\'{e} Laval, D\'{e}partement de Math\'{e}matiques et de Statistique, 1045 avenue de la M\'{e}decine,
Qu\'{e}bec, QC, G1V0A6, Canada.  maxime.breden@ens-cachan.fr}
\and Jean-Philippe Lessard \footnote{Universit\'{e} Laval, D\'{e}partement de Math\'{e}matiques et de Statistique, 1045 avenue de la M\'{e}decine,
Qu\'{e}bec, QC, G1V0A6, Canada. jean-philippe.lessard@mat.ulaval.ca.}
\and Maxime Murray \footnote{Florida Atlantic University, Department of Mathematical Sciences,
         Science Building, Room 234, 777 Glades Road,
         Boca Raton, Florida, 33431 , USA. mmurray2016@fau.edu}
         }

\date{}

\maketitle

\begin{abstract}
In this paper, we prove existence of symmetric homoclinic orbits for the suspension bridge equation $u''''+\beta u'' + e^u-1=0$ for all parameter values $\beta \in [0.5,1.9]$. 
For each $\beta$, a parameterization of the stable manifold is computed and the symmetric homoclinic orbits are obtained by solving a projected boundary value problem using Chebyshev series. The proof is computer-assisted and combines the uniform contraction theorem and the radii polynomial approach, which provides an efficient means of determining a set, centered at a numerical approximation of a solution, on which a Newton-like operator is a contraction.
\end{abstract}

\begin{center}
\end{center}


\section{Introduction} 
\label{sec:introduction}

One of the simplest models~\cite{McKennaWalter,MR3119065} for a suspension bridge is the partial differential equation (PDE)
\begin{equation} \label{eq:bridge}
\frac{\partial^2 U}{\partial T^2} = -\frac{\partial^4 U}{\partial X^4} - e^U +1.
\end{equation}
Here $U(T,X)$ describes the deflection of the roadway from the rest state $U=0$ as a function of time~$T$ and the spatial variable~$X$ (in the direction of traffic).
This paper is concerned with traveling wave solutions of~\eqref{eq:bridge}, i.e.,  solutions $U(T,X)=u(X-cT)$ describing a disturbance with profile $u$ propagating at velocity $c$ along the surface of the bridge. In particular, we apply a computer-assisted proof method to show that there is a large range of velocities for which such a solitary wave exists. 

Looking for traveling waves of \eqref{eq:bridge} with wave speed $c$ leads to the ordinary differential equation
\begin{equation}\label{eq:ode}
u''''+ c^2 u'' + e^u-1=0 .
\end{equation}
For large positive and negative values of the independent variable $t=X-cT$ we assume the solution to converge to the equilibrium $u=0$.
Due to the reversibility symmetry of the PDE in both time and space, we may restrict our attention to symmetric solutions. Hence, setting $\beta=c^2$, we are looking for symmetric homoclinic orbits satisfying
\begin{equation} \label{eq:brigdeODE}
\left\{ \begin{array}{l}
u''''+ \beta u''  + e^u-1=0 \\
u(-t)=u(t) \\
\lim_{t \to \infty} u(t) = 0.
\end{array}\right.
\end{equation}

Fourth order differential equations of the form $u''''+\beta u'' +f(u)=0$ for various nonlinearities $f$ have been studied extensively. 
For the bistable nonlinearity $f(u)=u^3-u$ the equation is a standard model in pattern formation, called the Swift-Hohenberg equation (see~\cite{MR1839555} and references therein), whereas the quadratic nonlinearity $f(u)=u^2-u$ appears, for example, in the study of water waves~\cite{MR1388874}. For the piecewise linear case $f(u)= \max\{ u,0\} $ 
homoclinic solutions were obtained in~\cite{McKennaWalter, ChenMcKenna}. For the problem with the exponential nonlinearity $f(u)=e^u-1$ a family of periodic solutions was established in~\cite{MR1632819}.

In \cite{ChenMcKenna} the question about existence of a symmetric homoclinic orbit of \eqref{eq:brigdeODE} is raised. This question was addressed by variational methods in \cite{MR1929147}, where the authors proved the result for \emph{almost all}  parameter values $\beta \in (0,2)$. In \cite{MR2578798} the existence of homoclinic orbits was demonstrated for all $\beta \in (0,c_*^2) \approx (0,0.5516)$, again using variational methods as well as intricate estimates on the second variation. In a different direction, using a computer-assisted proof, it was proven in \cite{MR2220064} that \eqref{eq:brigdeODE} has at least 36 homoclinic solutions for the single parameter value $\beta=1.69$.

In the present paper we complement the above results by proving the following.
\begin{theorem} \label{thm:main_result}
For all parameter values $\beta\in[0.5,1.9]$ there exists a symmetric homoclinic orbit of \eqref{eq:brigdeODE}.
\end{theorem}

We remark that for $|\beta| < 2$ the origin is a saddle-focus, while for $\beta > 2$ it is a saddle-center. Furthermore, we note the integral identity $\int_{\mathbb{R}} |u''|^2 - \beta |u'|^2 =  - \int_{\mathbb{R}} (e^u-1)u $. Since the right hand side is non-positive, homoclinic orbits are excluded for $\beta \leq 0$.
It is thus expected that the parameter range for which homoclinics exist is $\beta \in (0,2)$, or, equivalently, wave speeds $c\in(0,\sqrt{2})$. 
Our method for proving the result in Theorem~\ref{thm:main_result} is computer-assisted. While it can certainly be extended somewhat beyond the interval $[0.5,1.9]$, it is not possible to cover the entire range $(0,2)$ in this way. Indeed, as  $\beta$ decreases towards $0$ the amplitude of the solution diverges ($u$ becomes very negative), whereas when $\beta$ tends to $2$ the  homoclinic orbit collapses onto the trivial solution. In both limit regimes computer-assisted proofs become harder and harder. Since the result in~\cite{MR2578798} already covers the range $\beta \in (0,0.55]$, we thus focus on the parameter range $[0.5,1.9]$. We note that at $\beta=2$ a Hamiltonian-Hopf bifurcation occurs. In future work we intend to unfold this bifurcation and subsequently connect the homoclinic orbit that bifurcates to the branch covered by Theorem~\ref{thm:main_result} (at that point we will know how far we have to push the current continuation technique beyond $\beta=1.9$ to connect all the way to the bifurcation point).

The rest of the paper is dedicated to the proof of Theorem~\ref{thm:main_result}. 
Our approach begins by rewriting \eqref{eq:brigdeODE} as a first order system
for $(u_1,u_2,u_3,u_4)= (u,u',u'',u''')$ and then making the change of variables $(v_1,v_2,v_3,v_4)=(e^{u_1}-1,u_2,u_3,u_4)$ to obtain
\begin{equation}
\label{quadratic_system}
\left\{
  \begin{aligned}
  & v_1'=v_2+v_1v_2\\
  & v_2'=v_3\\
  & v_3'=v_4\\
  & v_4'=-\beta v_3-v_1.
  \end{aligned}
\right.
\end{equation}
There are two reasons for performing this change of variables. First, it turns the system into a polynomial vector field, which has technical advantages when performing the analysis to derive the necessary bounds. Second, while $u_1$ may become very negative for small values of $\beta$, the variable $v_1$ is always bounded from below by $-1$.
Our goal is now to prove the existence of symmetric homoclinic solutions to \eqref{quadratic_system} for all $\beta\in[0.5,1.9]$. 

We split the problem into two parts. On the one hand a rigorous computational description of the local (un)stable manifold is required. On the other hand we need to solve, via a rigorous computational technique, a boundary value problem for the part of the orbit between the local invariant manifolds. We attack both parts by a continuation technique in the context of the radii polynomial approach. This parametrized Newton-Kantorovich method, adapted to a computational setting, is introduced in Section~\ref{sec:radii_poly_approach}. 
In Section~\ref{Manifold} we combine this with the parameterization method to obtain descriptions of the local (un)stable manifold of the equilibrium $0 \in \R^4$. Essentially the same technique is then applied in Section~\ref{orbit} in a Chebyshev series setting to solve the boundary value problem. These two aspects are then combined into a rigorous computational continuation of the homoclinic solution to~\eqref{eq:brigdeODE}. 
We note that for smaller values of $\beta$ the boundary value problem is the more difficult part of the problem, as the orbit makes a bigger and bigger excursion away from the origin. On the other hand, for values of $\beta$ close to $2$ it is more difficult to obtain the local (un)stable manifold of the origin, as the real part of the eigenvalues tends to $0$. 
The algorithmic issues encountered when implementing the proof of Theorem~\ref{thm:main_result} are discussed in Section~\ref{s:algorithm}.

Finally, let us mention that there is a growing literature on the subject of computer-assisted methods for proving existence of connecting orbits, see \cite{SzczelinaZgli,BDLM,BLMM,Wilczak,Wilczak2,WilczakZgli,WojcikZgli}.
The main novel contribution of the current paper is to do rigorous continuation of a homoclinic orbit over a large range of parameter values. The method is generally applicable for connecting orbits problems in parameter dependent problems. In that sense Theorem~\ref{thm:main_result}, while providing a new  result for traveling waves in the suspension bridge problem which complements earlier work, is an illustration.


\section{The radii polynomial approach} 
\label{sec:radii_poly_approach}

In this section we present the functional analytic setup of our continuation method, which is formulated in terms of the {\em radii polynomials}, see Definition~\ref{PolRayon}. It will be used both to find the stable manifold and to solve the boundary value problem. For more details and proofs we refer to~\cite{MR3125637,MR2338393,MR2630003}.

Consider a sequence of Banach spaces $\left( X_1, \|\cdot \|_{X_1} \right), \dots , \left( X_d, \|\cdot \|_{X_d} \right)$ and the (product) Banach space
\[
X = X_1 \times  X_2 \times \cdots \times X_d,
\] 
with the induced norm defined by
\[
\| x \|_X = \max \left(\| x^{(1)} \|_{X_1} , \dots, \| x^{(d)} \|_{X_d} \right),
\]
where $x=(x^{(1)},\dots,x^{(d)}) \in X$ with $x^{(j)} \in X_j$ for $j=1,\dots,d$. Denote by 
\[
B_r(y) = \{ x \in X \mid \|x-y \|_X \le r \}
\]
the closed ball of radius $r>0$ centered at $y \in X$.

Consider an interval of parameters $[\beta_0,\beta_1] \subset \mathbb R$ and $T:[\beta_0,\beta_1]\times X \to X$ a Fr\'echet differentiable operator. For each $j=1,\ldots,d$, denote by $T^{(j)} :[\beta_0,\beta_1]\times X \to X_j$ the projection of $T$ onto $X_j$. Let $\bx_{\beta_0},\bx_{\beta_1} \in X$ be approximate fixed points of $T(\beta_0,\cdot)$ and $T(\beta_1,\cdot)$, respectively, and define the linear interpolation
\begin{equation} \label{eq:centers}
\bx_\beta \bydef \frac{\beta_1-\beta}{\beta_1-\beta_0}\bx_{\beta_0} + \frac{\beta-\beta_0}{\beta_1-\beta_0}\bx_{\beta_1}.
\end{equation}
Define the line of \emph{centers} by $\{ \bx_\beta \mid \beta \in [\beta_0,\beta_1]\} \subset X$. For each $j=1,\dots,d$, define the bounds
\begin{alignat}{1}
	\sup_{\beta \in [\beta_0,\beta_1]} \left\| T^{(j)}(\beta,\bx_\beta) - \bx^{(j)}_\beta \right\|_{X_j} & \leq Y^{(j)} , \label{eq:Y_Bounds} \\
 \sup_{\substack{b,c \in B_r(0) \\ \beta \in[\beta_0,\beta_1]}} \left\| D_xT^{(j)}(\beta,\bx_\beta +b)c \right\|_{X_j}   & \leq Z^{(j)}(r), \label{eq:Z_Bounds}
\end{alignat}
for some $Y^{(j)} >0 $ and $Z^{(j)}: \R^+ \mapsto \R^+:r \to Z^{(j)}(r)$.
The goal of the radii polynomial approach is to provide an efficient way to prove that an operator is a uniform contraction over a subset of $X$. This subset consists of small balls around the line of centers, provided by the linear interpolation between two numerical approximations of solutions at different parameter values. 
\begin{definition} 
Let $X$ be a Banach space and $B \subset X$.	
Let $[\beta_0,\beta_1] \subset \mathbb R$ be a set of parameters. A function $\tilde T : [\beta_0,\beta_1]\times B \to B$ is a \emph{uniform contraction} if there exists a 
constant $\kappa$ such that $0 < \kappa < 1$ and such that 
 $\| \tilde T(\beta,x) - \tilde T(\beta,y) \|_X \le \kappa \| x - y \|_X$ for all  $x, y \in B$ and all $\beta \in [\beta_0,\beta_1]$.
\end{definition}

The following result is a restatement of the uniform contraction principle (e.g. see \cite{MR660633} for a proof).

\begin{theorem}[\bf Uniform Contraction Principle] \label{th_local}
Suppose there exists some $r>0$ such that 
\begin{equation} \label{eq:uniform_T_s}
\tilde T :
\left\{
\begin{aligned}
\left[\beta_0,\beta_1\right] \times B_r(0) &\longrightarrow B_r(0) \\
(\beta,x) &\longmapsto \tilde T(\beta,x) \bydef T(\beta,x+\bx_\beta )-\bx_\beta
\end{aligned}
\right.
\end{equation}
is a uniform contraction, then for every $\beta \in [\beta_0,\beta_1]$, there exists a unique $\tilde x(\beta) \in B_{r}(\bx_\beta)$ such that $T(\beta,\tilde x(\beta))=\tilde x(\beta)$. Moreover, the function $\beta\mapsto \tilde x(\beta)$ is of class $C^{k}$ if $(\beta,x)\mapsto T(\beta,x)$ is of class $C^{k}$. 
\end{theorem}

With the bounds $Y$ and $Z$ on the residue and the derivative of $T$, see Equations~\eqref{eq:Y_Bounds} and~\eqref{eq:Z_Bounds}, contractivity can be checked explicitly. This is expressed by the next theorem (we refer to~\cite{MR3125637,MR2338393,MR2630003} for a proof).

\begin{theorem}\label{contr}
Given a set of parameters $[\beta_0,\beta_1]\subset\mathbb{R}$, consider the set of centers $\{ \bx_\beta \mid \beta \in [\beta_0,\beta_1]\}$ with $\bx_\beta$ given by \eqref{eq:centers}.
Assume that $T: [\beta_0,\beta_1] \times X \to X$ is an operator satisfying the bounds  \eqref{eq:Y_Bounds} and \eqref{eq:Z_Bounds}. If 
there exists $r>0$ such that $Y^{(j)} + Z^{(j)}(r)<r$, for each $j=1,\dots,d$, then $\tilde T$, defined by \eqref{eq:uniform_T_s}, is a uniform contraction
(on $B_r(0)$).
\end{theorem}

Assuming we have determined explicit bounds $Y^{(j)}$ and $Z^{(j)}(r)$, where in practice the latter is a polynomial with positive coefficients, satisfying \eqref{eq:Y_Bounds} and \eqref{eq:Z_Bounds}. It is convenient to introduce the radii polynomials, which provide an efficient way in verifying the hypotheses of Theorem~\ref{contr}.
\begin{definition}\label{PolRayon}
Let $Y=(Y^{(1)},\dots,Y^{(d)})$ and $Z=(Z^{(1)},\dots,Z^{(d)})$ be the bounds on the operator $T_\beta$ as given by \eqref{eq:Y_Bounds} and \eqref{eq:Z_Bounds} respectively. We define the radii polynomials as
\begin{equation} \label{eq:def_rad_poly_general}
	p_j(r)\bydef Y^{(j)} + Z^{(j)}(r) -r, \quad j=1,\dots,d.
\end{equation}
\end{definition}
One can see that the radii polynomials depend on the upper bounds $Y$ and $Z$, and therefore they are not uniquely defined. 
But the smaller these bounds are, the higher the chances are to prove that the operator $T_\beta$ is a contraction over a ball around the approximation. 
The following result shows how the radii polynomials are used in practice to give us the value of $r$ for which we can apply Theorem~\ref{contr}.

\begin{proposition}\label{prop:Radii}
Let
\[
	\mathcal{I} \bydef \bigcap_{j=1}^d \lbrace r>0 \mid p_j(r)<0 \rbrace,
\]
and assume that $\mathcal{I} \neq \emptyset$. Then $\mathcal{I}$ is an open interval of $\mathbb R^+$, i.e., $\mathcal{I} = (r_{\rm min},r_{\rm max})$. 
For any $r_0 \in (r_{\rm min},r_{\rm max})$, $\tilde T: B_{r_0}(0) \times \left[\beta_0,\beta_1\right] \longrightarrow B_{r_0}(0)$ is a uniform contraction.
\end{proposition}


\section{Parameterization of the stable manifold}\label{Manifold}

In this section we compute an approximate parameterization of the (local) stable manifold at 0, and provide explicit error bounds on this parameterization. This is done by combining the ideas of the \textit{parameterization method} (first introduced in \cite{parameterization_method,MR1976080,MR2177465}, see also~\cite{MR3467671}) and of \textit{rigorous computation} (following the approach of \cite{BMR,maximizing_manifold}). Having computed the parameterization, we will be able to obtain the homoclinic connection in the next section by taking advantage of the fact that it is now enough to compute an orbit on a finite time interval, i.e., an orbit that ends up in the local stable and unstable manifolds (or rather, we compute and verify an orbit that starts from the symmetric section and ends up, after some finite time, in the local stable manifold, see~\eqref{eq:brigdeODE}).

\subsection{Looking for the stable manifold as a zero finding problem \boldmath$F(\beta,a)=0$\unboldmath}
\label{s:manintro}

The first step is to recast the problem of finding a parameterization as
looking for a zero of a map $F$, which is the aim of this section. Setting
\begin{equation*}
	\Psi_{\beta}(v) \bydef \begin{pmatrix} v_2+v_1v_2\\ v_3\\ v_4\\-\beta
v_3-v_1\end{pmatrix},
\end{equation*}
Equation~\eqref{quadratic_system} is rewritten as
$v'=\Psi_{\beta}(v)$. The Jacobian at the origin is  
\begin{equation*}
D\Psi_{\beta}(0)=
  \begin{pmatrix}
  0 & 1 & 0 & 0\\
  0 & 0 & 1 & 0\\
  0 & 0 & 0 & 1\\
  -1 & 0 & -\beta & 0
  \end{pmatrix},
\end{equation*}
and one finds that for $\beta\in[0,2)$ it has two complex conjugated eigenvalues with negative real part, which we denote by $\lambda(\beta)$ and $\lambda^*(\beta)$:
\begin{equation}
\label{eq:lambda}
\lambda(\beta) = -\frac{1}{2}\sqrt{2-\beta} + \text{i} \frac{1}{2}\sqrt{2+\beta} .
\end{equation}
The associated eigenvectors are given by $V(\beta)$ and $V^*(\beta)$, where
\begin{equation}
\label{eq:V}
V(\beta)=\begin{pmatrix}
1\\ \lambda(\beta) \\ \lambda(\beta)^2 \\ \lambda(\beta)^3
\end{pmatrix}.
\end{equation}

The stable manifold at $0$ is thus two dimensional. Since $\Psi_{\beta}$ is analytic we may look for an analytic local parameterization of this manifold. We will look for this parameterization as a power series
\begin{equation}
\label{Q}
Q_{\beta}(\theta)=\sum_{\vert\alpha\vert\geq 0}a_{\alpha}(\beta)\theta^{\alpha},\quad 
\theta=\begin{pmatrix}\theta_1\\ \theta_{2}\end{pmatrix} \in {\mathbb{C}}^{2},
\ a_{\alpha}(\beta) = \begin{pmatrix}a_{\alpha}^{(1)}(\beta) \\  a_{\alpha}^{(2)}(\beta) \\  a_{\alpha}^{(3)}(\beta) \\  a_{\alpha}^{(4)}(\beta)\end{pmatrix} \in\mathbb{C}^4,
\end{equation}
with standard multi-index notation: $\alpha\in\mathbb{N}^2$, $\vert \alpha\vert =\alpha_1+\alpha_2$, $\theta^{\alpha}=\theta_1^{\alpha_1}\theta_2^{\alpha_2}$, and satisfying
\begin{equation}
\label{eq:additional_conditions}
Q_{\beta}(0)=0,\quad DQ_{\beta}(0)=\begin{pmatrix} V(\beta) & V^*(\beta) \end{pmatrix},
\end{equation}
together with the invariance equation
\begin{equation}
\label{eq:invariance_equation}
DQ_{\beta}(\theta)\begin{pmatrix} \lambda(\beta) & 0\\ 0 & \lambda^*(\beta) \end{pmatrix} \theta = \Psi_{\beta}(Q_{\beta}(\theta)).
\end{equation}

\begin{remark}
Even though the vector field $\Psi_{\beta}$ is real, the fact that we have two complex eigenvalues makes it easier to first look for a parameterization $Q_{\beta}$ of the complex manifold and then recover the real parameterization (the one which will be of interest in the next section for computing the homoclinic orbit) by considering
\begin{equation*}
P_{\beta}(\theta) \bydef Q_{\beta}(\theta_1+ \rm{i} \theta_2,\theta_1- \rm{i} \theta_2), \quad \text{for }\theta\in\mathbb{R}^2,
\end{equation*}
see \cite{MirelessMischaikow,BMR} for more details. This is due to the underlying symmetry
$a_{(\alpha_2,\alpha_1)}=a^*_{(\alpha_1,\alpha_2)}$, which is respected by the function $F$ introduced below.
\end{remark}

Plugging the power series~\eqref{Q} into the invariance equation \eqref{eq:invariance_equation} we get 
\begin{equation}
\label{eq:invariance_in_power_series}
\sum_{\vert \alpha \vert \geq 0} (\alpha_1 \lambda(\beta) + \alpha_2\lambda^*(\beta))a_{\alpha}(\beta)\theta^{\alpha} = \sum_{\vert \alpha\vert \geq 0} 
\begin{pmatrix}
a_{\alpha}^{(2)}(\beta) + (a^{(1)}(\beta)\star a^{(2)}(\beta))_{\alpha} \\
a_{\alpha}^{(3)}(\beta) \\
a_{\alpha}^{(4)}(\beta) \\
-a_{\alpha}^{(1)}(\beta) -\beta a_{\alpha}^{(3)}(\beta)
\end{pmatrix}
\theta^{\alpha},
\end{equation}
where $\star$ stands for the Cauchy product. We recall that, given two sequences $u$ and $v$ of complex numbers (indexed over $\mathbb{N}^2$), their Cauchy product is the sequence defined by
\begin{equation*}
(u\star v)_{\alpha}=\sum_{0\leq \sigma \leq \alpha} u_{\sigma}v_{\alpha-\sigma}, \qquad\text{for all }\alpha\in\mathbb{N}^2,
\end{equation*}
where $\sigma\leq \alpha$ means $\sigma_1\leq\alpha_1$ and $\sigma_2\leq\alpha_2$ (and similarly $\alpha-\sigma=(\alpha_1-\sigma_1,\alpha_2-\sigma_2)$).

Notice that the additional conditions \eqref{eq:additional_conditions} imply that the coefficients of total degree 0 and 1 are equal on both sides of Equation~\eqref{eq:invariance_in_power_series}.

Finding an analytic parameterization of the local manifold is now equivalent to find a zero of $F(\beta,\cdot)$, defined component-wise by
\begin{equation*}
F(\beta,a) =
\left\{
\begin{aligned}
& a_{0,0}  \quad &&\text{if }\alpha=(0,0), \\
& a_{1,0} - V(\beta) \quad &&\text{if }\alpha = (1,0), \\
& a_{0,1} - V^*(\beta) \quad &&\text{if }\alpha = (0,1), \\
& (\alpha_1 \lambda(\beta) + \alpha_2\lambda^*(\beta))a_{\alpha} - 
\begin{pmatrix}
a_{\alpha}^{(2)} + (a^{(1)}\star a^{(2)})_{\alpha} \\
a_{\alpha}^{(3)} \\
a_{\alpha}^{(4)} \\
-a_{\alpha}^{(1)} -\beta a_{\alpha}^{(3)}
\end{pmatrix} \quad&&\text{for } \vert\alpha\vert\geq 2.
\end{aligned}
\right.
\end{equation*}

\subsection{Getting to the fixed point formulation}
\label{s:fixedpointmanifold}

Let $\nu \geq 1$
and denote by $\ellnu$ the Banach space of complex valued sequences $u=\left(u_\alpha\right)_{\vert\alpha\vert\geq 0}$ such that
\begin{equation*}
\left\Vert  u \right\Vert_{1,\nu} \bydef \sum_{\vert\alpha\vert=0}^\infty \vert u_\alpha\vert \nu^{\vert\alpha\vert}< \infty .
\end{equation*}
This space is a Banach algebra under the Cauchy product, which gives us control on the quadratic terms. 
\begin{lemma}\label{LemmaNorm_cauchy}
For $u,v \in \ellnu$, $\left\Vert u \star v \right\Vert_{1,\nu} \leq \left\Vert u \right\Vert_{1,\nu} \left\Vert v \right\Vert_{1,\nu} $.
\end{lemma}
\begin{definition}
In this section we consider 
\begin{equation*}
X\bydef (\ell^1_{\nu})^4, \quad \text{with the norm }\left\Vert a \right\Vert_{X}\bydef \max\limits_{j=1,\ldots,4} \bigl\Vert a^{(j)} \bigr\Vert_{1,\nu}.
\end{equation*}
\end{definition}
We are going to look for zeros $a$ of $F(\beta,\cdot)$ in the space $X$. Notice that $a\in X$ means that
\begin{equation*}
\sum_{\vert\alpha\vert\geq 0} \bigl\vert a^{(j)}_{\alpha} \bigr\vert  \nu^{\vert\alpha\vert} < \infty,\quad \text{for } j=1,\ldots,4,
\end{equation*}
which ensures that the associated parameterization $Q_{\beta}$ is well defined at least for 
\begin{equation*}
\left\vert \theta\right\vert_{\infty} \bydef \max\left(\vert\theta_1\vert,\vert\theta_2\vert\right) \leq \nu.
\end{equation*}

We now explain how to rigorously determine a parameterization of the manifold for all values of $\beta$ in a given interval $[\beta_0,\beta_1]$. It will be more convenient to work with a rescaled parameter $s$ ranging between $0$ and $1$. Therefore we define  
\begin{equation*}
\beta_s=\beta_0+s(\beta_1-\beta_0)=\beta_0+s\Delta\beta, \quad \text{for } s\in[0,1].
\end{equation*}
We want to get a parameterization $a(s)$ such that 
\begin{equation*}
F(\beta_s,a(s))=0,\quad \text{for } s\in[0,1].
\end{equation*}
Since we are working on the interval $[\beta_0,\beta_1]$, parameterized by $s \in [0,1]$, we have altered the notation of the parametrization of the coefficients $a$ slightly compared to Section~\ref{s:manintro}, namely $a(s)$ instead of $a (\beta)$.
We will use $a(s)$ throughout the remainder of the paper, except in Section~\ref{s:rewriting}, where the notation $a(\beta)$ is more appropriate.

We first compute approximate zeros $\overa(0)$ and $\overa(1)$ of $F(\beta_0,\cdot)$ and $F(\beta_1,\cdot)$ respectively, by solving numerically the truncated problem (for $s=0$ and $s=1$)
\begin{equation*}
F^{[N]}(\beta_s,\cdot) \bydef \left(F_{\alpha}(\beta_s,\cdot)\right)_{0\leq\vert\alpha\vert<N}=0,
\end{equation*} 
for some $N \geq 1$,
and by padding the obtained solutions with 0 to get elements of $X=(\ell^1_{\nu})^4$. We then define for $s\in[0,1]$
\begin{equation*}
\overa(s) \bydef \overa(0)+s(\overa(1)-\overa(0))=\overa(0)+s\Delta\overa.
\end{equation*}
If $\overa(0)$ and $\overa(1)$ are two good approximate zeros (of $F(\beta_0,\cdot)$ and $F(\beta_1,\cdot)$ respectively) and if $\vert\beta_1-\beta_0\vert$ is not too large, $\overa(s)$ should be a good approximate zero of $F(\beta_s,\cdot)$ for each $s\in[0,1]$. We are going to reformulate this claim into a mathematical statement and prove that in a given neighbourhood of $\overa(s)$ there exist a unique zero of $F(\beta_s,\cdot)$ for all $s\in[0,1]$. 
To put this in the framework described in Section~\ref{sec:radii_poly_approach}, we consider the operator 
\begin{equation*}
T(\beta,a)=a-AF(\beta,a),
\end{equation*}
where $A$, defined below, is an approximate inverse of $D_aF(\beta_0,\overa(0))$. Namely, for $N$ large enough,
\begin{equation*}
A^{\dag} \bydef
\begin{pmatrix}
D_aF^{[N]}(\beta_0,\overa(0)) & & 0 & \\
 & \tilde{M}_N & & \\
0 & & \tilde{M}_{N+1} & \\
 & & & \ddots\\
\end{pmatrix}
\end{equation*}
should be a reasonably good approximation of $D_aF(\beta_0,\overa(0))$, where, for any $k\geq N$, $\tilde{M}_k$ is the $4(k+1)$ by $4(k+1)$ block diagonal matrix 
\begin{equation*}
\tilde{M}_k \bydef
\begin{pmatrix}
k\lambda(\beta_0) I_4 & & 0 & \\
 & ((k-1)\lambda(\beta_0)+\lambda^*(\beta_0))I_4 & & \\
0 & & \ddots & \\
 & & & k\lambda^*(\beta_0) I_4
\end{pmatrix},
\end{equation*}
with $I_4$ the 4 by 4 identity matrix. Finally, we define $A$ as
\begin{equation}\label{e:defA}
	A \bydef
\begin{pmatrix}
J & & 0 & \\
 & M_N & & \\
0 & & M_{N+1} & \\
 & & & \ddots\\
\end{pmatrix},
\end{equation}
where $J$ is a numerical approximation of $\left(D_aF^{[N]}(\beta_0,\overa(0))\right)^{-1}$, while the $M_k=\tilde{M}_k^{-1}$ are exact inverses. 
The operators $A^\dagger$ and $A$ are then approximate inverses of each other: approximate in the finite part and exact in the infinite tail.
\begin{remark}\label{r:Aijalpha}
To make sense of this matrix representation of $A^\dagger$ and $A$, as well as $\tilde{M}_k$ and $M_k$, one should think of $a_\alpha$ as an infinite vector where the elements are ordered according to increasing degree $\vert\alpha\vert = \alpha_1+\alpha_2$ and within fixed degree by increasing $\alpha_2$, while also taking into account that each $a_{\alpha}$ is a vector in $\mathbb{C}^4$. This means that $a$ is represented as
\begin{equation*}
a=\begin{pmatrix}
a_{0,0} \\ a_{1,0} \\ a_{0,1} \\ a_{2,0} \\ a_{1,1} \\ a_{0,2} \\ \vdots
\end{pmatrix}, 
\text{ where }  
a_{\alpha}= \begin{pmatrix}
a_{\alpha}^{(1)} \\ a_{\alpha}^{(2)} \\ a_{\alpha}^{(3)}  \\ a_{\alpha}^{(4)}
\end{pmatrix}
\text{ for each }\alpha\in\mathbb{N}^2, \text{ and that }
a^{(j)}=\begin{pmatrix}
a_{0,0}^{(j)} \\ a_{1,0}^{(j)} \\ a_{0,1}^{(j)} \\ a_{2,0}^{(j)} \\ a_{1,1}^{(j)} \\ a_{0,2}^{(j)} \\ \vdots
\end{pmatrix}
\text{ for }j=1,\ldots,4.
\end{equation*}
The above representation describes the operators as infinite matrices 
where each element $A_{\alpha',\alpha}$ is a linear operator on $\mathbb{C}^4$,
i.e.\ a $4\times4$ matrix that we will occasionally denote by $A_{\alpha',\alpha}=\{A_{\alpha',\alpha}^{(i,j)} \}_{1\leq i,j \leq 4}$.
\end{remark}

We now follow the ideas described in Section~\ref{sec:radii_poly_approach}, using the Banach space $X=\left(\ell^1_{\nu}\right)^4$ endowed with the norm $\Vert a\Vert_X=\max\limits_{j=1,\ldots,4} \Vert a^{(j)}\Vert_{1,\nu}$. In the next subsections we are going to compute the bounds $Y^{(j)}$ and $Z^{(j)}(r)$ and the associated radii polynomials, and then prove that for some positive $r$ each radii polynomial $p_j(r)$ is negative, which will yield (for each $s\in[0,1]$) the existence of a unique zero $a(\beta_s)$ of $F(\beta_s,\cdot)$ in the ball of radius $r$ around $\overa(s)$. At this point we will know that $\overa(s)$ defines an approximate parameterization of the stable manifold, with an error bound controlled by $r$. We will use this in Section~\ref{orbit} to prove the existence of a homoclinic orbit for all $\beta\in[\beta_0,\beta_1]$. Moreover, derivatives of the manifold with respect to $\theta$, which will be needed in Section~\ref{s:rewriting}, can also be approximately computed with rigorous control on the error bound, see Lemma~\ref{bornederive}. 

\subsection{The bound \boldmath$Y$\unboldmath}

In this section we focus on the bound $Y$ defined in \eqref{eq:Y_Bounds}. Let $\vert A \vert$ denote the component-wise absolute value of $A$. In order to define the bound we are looking for, we try to bound every term $\left(T(\beta_s,\overa(s))-\overa(s)\right)_\alpha^{(j)}$ with $|\alpha|\geq 0$ and $j=1,2,3,4$:
\begin{align*} &\left|\left(T(\beta_s,\overa(s))-\overa(s)\right)^{(j)}_\alpha\right| 
 = \left|\big(AF(\beta_s,\overa(s))\big)^{(j)}_\alpha\right| \\
& \qquad \leq \Bigg(\left\vert A \right\vert \bigg(\left\vert F(\beta_0,\overa(0))\right\vert + \left\vert D_aF(\beta_0,\overa(0))\Delta\overa + D_{\beta}F(\beta_0,\overa(0))\Delta\beta\right\vert \\
& \qquad \qquad + \frac{1}{2}\max\limits_{s\in[0,1]}\left\vert D^2_{aa} F(\beta_s,\overa(s))(\Delta\overa)^2 + 2D^2_{a\beta} F(\beta_s,\overa(s))\Delta\overa\Delta\beta 
+ D^2_{\beta\beta} F(\beta_s,\overa(s))(\Delta\beta)^2 \right\vert\bigg)\Bigg)_{\alpha}^{(j)}.
\end{align*}
A straightforward calculation (using that $\vert\lambda(\beta)\vert=1$ and computing the derivatives of $\lambda$ and $V$ with respect to $\beta$) yields that, for all $\vert\alpha\vert\geq 0$,
\begin{equation*}
\frac{1}{2}\max\limits_{s\in[0,1]}\left\vert D^2_{aa} F_{\alpha}(\beta_s,\overa(s))(\Delta\overa)^2 + 2D^2_{a\beta} F_{\alpha}(\beta_s,\overa(s))\Delta\overa\Delta\beta + D^2_{\beta\beta} F_{\alpha}(\beta_s,\overa(s))(\Delta\beta)^2 \right\vert \leq G_{\alpha},
\end{equation*}
where
\begin{equation*}
G_{\alpha} \bydef \left\{
\begin{aligned}
& 0 \quad && \alpha=(0,0), \\
& \frac{1}{2}\left(\frac{1}{4}\sqrt{\frac{4+3\beta_1^2}{(2-\beta_1)^3(2+\beta_1)^3}}\begin{pmatrix}
0\\
1\\
2\\
3
\end{pmatrix} +\frac{1}{4}\frac{1}{(2-\beta_1)(2+\beta_1)}\begin{pmatrix}
0\\
0\\
2\\
6
\end{pmatrix}\right)(\Delta\beta)^2 \quad && \vert\alpha\vert=1,\\
& \frac{1}{4}\sqrt{\frac{(\alpha_1+\alpha_2)^2}{2-\beta_1}+\frac{(\alpha_1-\alpha_2)^2}{2+\beta_0}}\Delta\beta\Delta\overa_{\alpha} +
\begin{pmatrix}
\left\vert\Delta\overa^{(1)} \star \Delta\overa^{(2)}\right\vert_{\alpha}\\
0\\
0\\
\left\vert \Delta\beta\Delta\overa_{\alpha}^{(3)}\right\vert
\end{pmatrix} \quad && \vert\alpha\vert\geq 2.
\end{aligned}
\right.
\end{equation*}

Since $\left( \overa(s)\right)_{\alpha}=0$ for all $\vert\alpha\vert\geq N$ and $F(\beta,\cdot)$ is quadratic in $a$, we have that  $F_{\alpha}(\beta_s,\overa(s))$ vanishes as soon as $\vert\alpha\vert\geq 2N-1$. Therefore, we define $\tilde F$ component-wise by
\begin{equation*}
\tilde F_{\alpha} = \left\{
\begin{aligned}
& \left\vert F(\beta_0,\overa(0))\right\vert_{\alpha} + \left\vert D_aF(\beta_0,\overa(0))\Delta\overa + D_{\beta}F(\beta_0,\overa(0))\Delta\beta\right\vert_{\alpha} + G_{\alpha} \quad && \vert\alpha\vert<2N-1 , \\
& 0 \quad && \vert\alpha\vert\geq 2N-1 ,
\end{aligned}
\right.
\end{equation*}
and then set 
\begin{equation*}
Y^{(j)}=\left\Vert \left(\vert A\vert \tilde F\right)^{(j)} \right\Vert_{1,\nu},
\end{equation*}
so that
\begin{equation*}
\left\| \left(T(\beta_s,\overa(s))-\overa(s)\right)^{(j)} \right\|_{1,\nu} \leq Y^{(j)} \qquad \text{for } j=1,\ldots,4, ~s\in[0,1].
\end{equation*}

\subsection{The bound \boldmath$Z$\unboldmath} \label{Zmanifold}

In this section we derive the bound $Z$ defined in \eqref{eq:Z_Bounds}.
Let $b,c \in B_r(0)$. We split $D_aT(\beta_s,\overa(s) +b)c$ in three terms which will be easier to bound separately. For each $j=1,\ldots,4$,
\begin{align*}
\left\Vert \left(D_aT(\beta_s,\overa(s) +b)c\right)^{(j)}\right\Vert_{1,\nu} &= \left\Vert \left(\left(I -AD_aF(\beta_s,\overa(s) + b)\right)c\right)^{(j)}\right\Vert_{1,\nu} \\
&\leq \left\Vert \left(\left(I-AA^{\dag}\right)c\right)^{(j)}\right\Vert_{1,\nu} + \left\Vert \left(A\left(D_aF(\beta_s,\overa(s) + b)- A^{\dag}\right)c\right)^{(j)}\right\Vert_{1,\nu} \\
&\leq \left\Vert \left(\left(I-AA^{\dag}\right)c\right)^{(j)}\right\Vert_{1,\nu} + \left\Vert \left(A\left(D_aF(\beta_s,\overa(s))- A^{\dag}\right)c\right)^{(j)}\right\Vert_{1,\nu} \\
& \hspace*{6cm} + \left\Vert \left(AD^2_{aa}F(\beta_s,\overa(s))(b,c)\right)^{(j)}\right\Vert_{1,\nu} \\
&\leq Z_0^{(j)} r +Z_1^{(j)} r + Z_2^{(j)} r^2.
\end{align*}
The bounds $Z_i \bydef \left(Z_i^{(1)},Z_i^{(2)},Z_i^{(3)},Z_i^{(4)} \right) \in \R^4$ ($i=0,1,2$) are given in the following subsections.

\subsubsection{The bound \boldmath$Z_0$\unboldmath}
\label{sec:Z0}

From the  definitions of $A$ and $A^{\dag}$ we get
\begin{equation*}
I-AA^{\dag}=
\begin{pmatrix}
I_{\frac{4N(N+1)}{2}}-JD_aF^{[N]}(\beta_0,\overa(0)) & & 0 & \\
 & & & \\
0 & & 0 & \\
 & & & \ddots\\
\end{pmatrix}.
\end{equation*}
The finite matrix $B=I_{2N(N+1)}-JD_aF^{[N]}(\beta_0,\overa(0))$ can be computed using interval arithmetic. To obtain the bound $Z_0$ we only need to compute the operator norm of $B$ (as acting on $(\ell^1_\nu)^4$). 
This is the content of the following lemma.
\begin{lemma}
\label{lem:operator_norm}
Let $h=\left(h_{\alpha}\right)_{\alpha\in\mathbb{N}^2}\in\ell^1_{\nu}$ (with $h_{\alpha}\in\mathbb{C}$ for all $\alpha$) and $\Gamma$ a linear operator acting on $\ell^1_{\nu}$. Then
\begin{equation*}
\sup\limits_{\left\Vert h\right\Vert_{1,\nu} =1}\left\Vert \Gamma h\right\Vert_{1,\nu} =\sup\limits_{\alpha\in\mathbb{N}^2} \frac{1}{\nu^{\vert\alpha\vert}} \sum_{\alpha' \in\mathbb{N}^2} |\Gamma_{\alpha',\alpha}|\nu^{\vert\alpha'\vert}.
\end{equation*} 
In particular, if $\Gamma$ consists in a finite block $\Gamma^{[N]}$ of size $N(N+1)/2\times N(N+1)/2$ and a diagonal tail $\left(\gamma_{\alpha}\right)_{\vert\alpha\vert\geq N}$
\begin{equation*}
\Gamma=
\begin{pmatrix}
\Gamma^{[N]} & & 0 & \\
 & \gamma_{N,0} & & \\
 0 & & \gamma_{N-1,1} & \\
 & & & \ddots \\
\end{pmatrix},
\end{equation*}
then 
\begin{equation*}
\sup\limits_{\left\Vert h\right\Vert_{1,\nu} =1}\left\Vert  \Gamma h\right\Vert_{1,\nu} =\max\left(\max\limits_{\vert\alpha\vert < N} \frac{1}{\nu^{\vert\alpha\vert}} \sum_{\vert\alpha' \vert < N} |\Gamma _{\alpha',\alpha}|\nu^{\vert\alpha'\vert}, \sup\limits_{\vert\alpha\vert \geq N} \vert \gamma_{\alpha}\vert\right).
\end{equation*} 
\end{lemma}
Hence, we define
\begin{equation}\label{e:defK}
K^{(i,j)}(B) \bydef  \max_{0 \leq \vert \alpha \vert <N} \frac{1}{\nu^{\vert\alpha\vert}} \sum_{\vert\alpha'\vert<N} |B_{\alpha',\alpha}^{(i,j)}|\nu^{\vert\alpha'\vert} ,
\end{equation}
with the notation $B_{\alpha',\alpha}^{(i,j)}$ introduced in Remark~\ref{r:Aijalpha},
and set
\begin{equation*}
	Z_0^{(i)}=\sum_{j=1}^4 K^{(i,j)}(B),
\end{equation*}
to obtain 
\begin{equation}\label{Z0manifold}
\left\Vert \left(\left(I-AA^{\dag}\right)c\right)^{(j)}\right\Vert_{1,\nu} \leq Z^{(j)}_0 r, \quad \text{for } j=1,\ldots,4.
\end{equation}

\subsubsection{The bound \boldmath$Z_1$\unboldmath}

This term is the most involved one to bound tightly, so again we split it into several parts that we bound separately. For each $j=1,\ldots,4$,
\begin{align*}
\left\Vert\Big( A \left(D_aF(\beta_s,\overa(s))-A^{\dag}\right)c\Big)^{(j)}\right\Vert_{1,\nu} \leq &\left\Vert \Big(\vert A\vert\left\vert \left(D_aF(\beta_0,\overa(s))-A^{\dag}\right) c\right\vert\Big)^{(j)}\right\Vert_{1,\nu}\\
&\qquad  + \left\Vert \Big(\vert A\vert \max\limits_{\eta\in[0,1]}\vert\Delta\beta\vert \left\vert D^2_{\beta a}F(\beta_{\eta},\overa(s))c\right\vert\Big)^{(j)}\right\Vert_{1,\nu} \\
 \leq &\left\Vert \Big(\vert A\vert\left\vert \left(D_aF(\beta_0,\overa(0))-A^{\dag}\right) c\right\vert\Big)^{(j)}\right\Vert_{1,\nu} \\
 &\qquad + \left\Vert \Big(\vert A\vert\left\vert D^2_{aa}F(\beta_0,\overa(0)) (\Delta\overa,c)\right\vert\Big)^{(j)}\right\Vert_{1,\nu} \\
 &\qquad\qquad + \left\Vert \Big(\vert A\vert \max\limits_{\eta\in[0,1]}\vert\Delta\beta\vert \left\vert D^2_{\beta a}F(\beta_{\eta},\overa(s))c\right\vert\Big)^{(j)}\right\Vert_{1,\nu}.
\end{align*}
Let us focus first on the first term. Since 
\begin{equation*}
D_aF^{[N]}(\beta_0,\overa(0))c^{[N]}=\left(D_aF(\beta_0,\overa(0))c\right)^{[N]},
\end{equation*}
we get that
\begin{equation*}
\left(\left(D_aF(\beta_0,\overa(0))- A^{\dag}\right)c\right)^{[N]}=0.
\end{equation*}
For the tail $\vert\alpha\vert\geq N$ we find
\begin{equation*}
d_{\alpha} \bydef
\left(\left(D_aF(\beta_0,\overa(0))- A^{\dag}\right)c\right)_{\alpha} = 
\begin{pmatrix}
c^{(2)}_{\alpha} + (\overa(0)^{(1)}\star c^{(2)})_{\alpha} + (\overa(0)^{(2)}\star c^{(1)})_{\alpha} \\
c^{(3)}_{\alpha}\\
c^{(4)}_{\alpha}\\
-c^{(1)}_{\alpha} -\beta_0 c^{(3)}_{\alpha}
\end{pmatrix} ,
\end{equation*}
which we estimate by
\begin{align*}
\left\Vert d_{\alpha}^{(1)} \right\Vert_{1,\nu} &\leq  \left(1 + \left\Vert\overa(0)^{(1)} \right\Vert_{1,\nu} + \left\Vert\overa(0)^{(2)} \right\Vert_{1,\nu} \right)r
\\
\left\Vert d_{\alpha}^{(2)} \right\Vert_{1,\nu} &\le  r
\\
\left\Vert d_{\alpha}^{(3)} \right\Vert_{1,\nu} &\le  r
\\
\left\Vert d_{\alpha}^{(4)} \right\Vert_{1,\nu} &\le  \left(1+\beta_0\right)r.
\end{align*}
Now we use Lemma~\ref{lem:operator_norm} again and from the fact that $\vert (n-k)\lambda(\beta_0) + k\lambda^*(\beta_0)\vert \geq n\vert\Re(\lambda(\beta_0))\vert = n \frac{\sqrt{2-\beta_0}}{2}$ we infer that
\begin{equation*}
\left\Vert \left(Ad\right)^{(j)} \right\Vert_{1,\nu} \leq \frac{2}{N\sqrt{2-\beta_0}} \left\Vert d^{(j)} \right\Vert_{1,\nu},
\end{equation*}
and we are done with the first term. For the second term
\begin{equation*}
D^2_{aa}F_{\alpha}(\beta_0,\overa(0))(\Delta\overa,c) =
\begin{pmatrix}
(\Delta\overa^{(1)}\star c^{(2)})_{\alpha} + (\Delta\overa^{(2)}\star c^{(1)})_{\alpha} \\
0\\
0\\
0
\end{pmatrix}. 
\end{equation*}
Again we use Lemma~\ref{lem:operator_norm} to obtain
\begin{equation*}
\left\Vert \left( \left| A \right| \left| D^2_{aa} F_0(\overa(0))(\Delta\overa,c)\right|\right)^{(j)}  \right\Vert_{1,\nu} \leq \left\{
\begin{aligned}
& \max\left(K^{(1,1)}(J),\tfrac{2}{N\sqrt{2-\beta_0}}\right)\left(\Vert\Delta\overa^{(1)} \Vert_{1,\nu} + \Vert\Delta\overa^{(2)} \Vert_{1,\nu}\right)r  &\quad j=1,\\
& K^{(j,1)}(J)\left(\Vert\Delta\overa^{(1)} \Vert_{1,\nu} + \Vert\Delta\overa^{(2)} \Vert_{1,\nu}\right)r & \quad  j=2,3,4,
\end{aligned}
\right.
\end{equation*}
where we recall that $J$ is the block of $A$ corresponding to the floating point data, see~\eqref{e:defA}, and $K^{(i,j)}$ is defined by~\eqref{e:defK}.
Finally, computing the derivative of $\lambda$ with respect to $\beta$, we get that
\begin{equation*}
\max\limits_{\eta\in[0,1]}\vert\Delta\beta\vert \left\vert D^2_{\beta a}F(\beta_{\eta},\overa(s))c\right\vert_{\alpha} \leq \left\{
\begin{aligned}
& 0 \quad & \vert\alpha\vert<2, \\
& \Delta\beta\left(\frac{\vert\alpha\vert}{2\sqrt{(2-\beta_1)(2+\beta_1)}}\left\vert c_{\alpha}\right\vert +
\begin{pmatrix}
0\\
0\\
0\\
\left\vert c_{\alpha}^{(3)}\right\vert
\end{pmatrix}\right) \quad & \vert\alpha\vert\geq 2.
\end{aligned}
\right.
\end{equation*}
Now we need the following lemma, which is a slightly modified version of Lemma~\ref{lem:operator_norm}.
\begin{lemma}\label{lem:modifiedversion}
Let $c=\left(c_{\alpha}\right)_{\alpha\in\mathbb{N}^2}\in\left(\ell^1_{\nu}\right)^4$. We denote by $e$ the vector such that for all $\alpha$, $e_{\alpha}=\vert\alpha\vert c_{\alpha}$. Then for all $1\leq j\leq 4$
\begin{equation*}
\left\Vert \left(A e\right)^{(i)} \right\Vert_{1,\nu} \leq \max\left(\sum_{j=1}^4 \tilde K^{(i,j)}(J),\frac{2}{\sqrt{2-\beta_0}}\right) r
\end{equation*}
where
\begin{equation*}
\tilde K^{(i,j)}(J) \bydef  \max\limits_{\vert\alpha\vert<N} \left(\frac{\vert\alpha\vert}{\nu^{\vert\alpha\vert}} \sum_{\vert\alpha'\vert<N}  \left\vert J^{(i,j)}_{\alpha',\alpha}\right\vert \nu^{\vert\alpha'\vert}\right).
\end{equation*}
\end{lemma}
Using Lemmas~\ref{lem:operator_norm} and~\ref{lem:modifiedversion}  we infer that
\begin{equation*}
\left\Vert \Big(\vert A\vert \max\limits_{\eta\in[0,1]}\vert\Delta\beta\vert \left\vert D^2_{\beta a}F(\beta_{\eta},\overa(s))c\right\vert\Big)^{(j)}\right\Vert_{1,\nu} \leq \left\{
\begin{aligned}
& \Delta\beta\left(\overline{K}_j + K^{(j,4)}(J)\right)r \quad & j=1,2,3 \\
& \Delta\beta\left(\overline{K}_4 + \max\left(K^{(4,4)}(J),\frac{2}{N\sqrt{2-\beta_0}}\right) \right)r \quad & j=4,
\end{aligned}
\right.
\end{equation*}
where 
\begin{equation*}
\overline{K}_i \bydef \frac{\max\left(\sum_{j=1}^4 \tilde K^{(i,j)}(J),\frac{2}{\sqrt{2-\beta_0}}\right)}{2\sqrt{(2-\beta_1)(2+\beta_1)}}.
\end{equation*}
Finally, putting everything together, we define
\begin{align*}
Z^{(1)}_1 & =  \frac{2\left(1 + \left\Vert\overa(0)^{(1)} \right\Vert_{1,\nu} + \left\Vert\overa(0)^{(2)} \right\Vert_{1,\nu}\right)}{N\sqrt{2-\beta_0}} + \max\left(K^{(1,1)}(J),\frac{2}{N\sqrt{2-\beta_0}}\right)\left(\left\Vert\Delta\overa^{(1)} \right\Vert_{1,\nu} + \left\Vert\Delta\overa^{(2)} \right\Vert_{1,\nu}\right) \\
& \hspace{11.8cm} + \Delta\beta\left(\overline{K}_1 + K^{(1,4)}(J)\right) ,
\\
Z^{(2)}_1  &= \left(\frac{2}{N\sqrt{2-\beta_0}} + K^{(2,1)}(J)\left(\left\Vert\Delta\overa^{(1)} \right\Vert_{1,\nu} + \left\Vert\Delta\overa^{(2)} \right\Vert_{1,\nu}\right) + \Delta\beta\left(\overline{K}_2 + K^{(2,4)}(J)\right) \right),
\\
Z^{(3)}_1 &= \left(\frac{2}{N\sqrt{2-\beta_0}} + K^{(3,1)}(J)\left(\left\Vert\Delta\overa^{(1)} \right\Vert_{1,\nu} + \left\Vert\Delta\overa^{(2)} \right\Vert_{1,\nu}\right) + \Delta\beta\left(\overline{K}_3 + K^{(3,4)}(J)\right) \right),
\\
Z^{(4)}_1 &= \left(\frac{2(1+\beta_0)}{N\sqrt{2-\beta_0}} + K^{(4,1)}(J)\left(\left\Vert\Delta\overa^{(1)} \right\Vert_{1,\nu} + \left\Vert\Delta\overa^{(2)} \right\Vert_{1,\nu}\right)\right. + \left. \Delta\beta\left(\overline{K}_4 + \max\left(K^{(4,4)}(J),\frac{2}{N\sqrt{2-\beta_0}}\right) \right) \right),
\end{align*}
so that
\begin{equation*}
\left\Vert \left(A\left(D_aF(\beta_s,\overa(s))- A^{\dag}\right)c\right)^{(j)}\right\Vert_{1,\nu} \leq Z^{(j)}_1 r \qquad \text{for } j=1,\ldots,4, ~s\in[0,1].
\end{equation*}

\subsubsection{The bound \boldmath$Z_2$\unboldmath}

Since 
\begin{equation*}
D^2_{aa}F_{\alpha}(\beta_s,\overa(s))(b,c) = \left\{
\begin{aligned}
& 0 \quad & \vert\alpha\vert\leq 1, \\
& \begin{pmatrix}
(b^{(1)}\star c^{(2)})_{\alpha} + (b^{(2)}\star c^{(1)})_{\alpha} \\
0\\
0\\
0
\end{pmatrix} \quad & \vert\alpha\vert\geq 2,
\end{aligned}
\right.
\end{equation*}
we directly use one more time Lemma~\ref{lem:operator_norm} and set
\begin{equation*}
Z^{(1)}_2 =2\max\left( K^{(1,1)}(J),\frac{2}{N\sqrt{2-\beta_0}}\right), \quad Z^{(2)}_2 = 2K^{(2,1)}(J), \quad 
Z^{(3)}_2 = 2K^{(3,1)}(J) \quad \text{and}\quad Z^{(4)}_2 = 2K^{(4,1)}(J),
\end{equation*}
so that
\begin{equation*}
\left\Vert \left(AD^2_{aa}F(\beta_s,\overa(s))(b,c)\right)^{(j)}\right\Vert_{1,\nu} \leq Z_2^{(j)}r^2  \qquad \text{for } j=1,2,3,4, ~s\in[0,1].
\end{equation*}

\subsection{Use of the uniform contraction principle and error bounds}
\label{sec:para_error}

Following~\eqref{eq:def_rad_poly_general}, we set
\begin{equation} \label{eq:radii_polynomials_manifold}
p^{(j)}(r) \bydef Y^{(j)}+\left(Z^{(j)}_0+Z^{(j)}_1-1\right)r+Z^{(j)}_2r^2, \quad \text{for } j=1,\ldots,4.
\end{equation}
If we find an $r>0$ such that $p^{(j)}(r)<0$ for all $j=1,\ldots,4$,  then according to Proposition~\ref{prop:Radii} we have validated the numerical approximation $\overa(s)$ of the local stable manifold for $\beta=\beta_s$, for every $s \in [0,1]$.
\begin{proposition}\label{prop:validation_manifold}
For every $s\in [0,1]$, let 
\begin{equation*}
\overline Q_{\beta_s}(\theta)=\displaystyle \sum_{|\alpha|=0}^{N-1} \overa_{\alpha}(s) \theta^{\alpha}
\end{equation*} 
be the approximate parameterization of the complex local stable manifold that we have computed (for $\beta=\beta_s$).
Assume that there exists an $r>0$ such that $p^{(j)}(r)<0$ for all $j=1,\ldots,4$. Then, for each $s\in[0,1]$, there exists a parameterization $Q_{\beta_s}$ of the complex local stable manifold (for $\beta=\beta_s$) of the form
\begin{equation*}
Q_{\beta_s}(\theta)=\displaystyle \sum_{|\alpha|=0}^{\infty}  a_{\alpha}(s) \theta^{\alpha},
\end{equation*}
which is well defined for all $\theta\in \mathbb{C}^2$ satisfying $\vert \theta\vert_{\infty} \leq \nu$. Let 
\begin{equation}\label{e:defherror}
	\hat{h}_{\beta_s}(\theta) \bydef Q_{\beta_s}(\theta)-\overline Q_{\beta_s}(\theta) ,
\end{equation}
then we have the error bound $\vert \hat{h}_{\beta_s}(\theta) \vert_{\infty} \leq r$
for all $\vert \theta\vert_{\infty} \leq \nu$.
These statements still hold true for the real (approximate and exact) local stable manifold, defined by 
\begin{align} 
\label{eq:P_beta_s}
P_{\beta_s}(\theta) & \bydef Q_{\beta_s}(\theta_1+ \rm{i} \theta_2,\theta_1- \rm{i} \theta_2)
\\ 
\label{eq:barP_beta_s}
\overline P_{\beta_s}(\theta) & \bydef \overline Q_{\beta_s}(\theta_1+ \rm{i} \theta_2,\theta_1- \rm{i} \theta_2)
\end{align}
for all $\theta\in \mathbb{R}^2$ satisfying $\vert \theta\vert_{2} \bydef \sqrt{\theta_1^2+\theta_2^2} \leq \nu$.
\end{proposition}
\begin{proof}
Proposition~\ref{prop:Radii} yields that, for each $s\in[0,1]$, there exists a
unique fixed point $a(s)$ of $T(\beta_s,\cdot)$ in the ball of radius $r$ around $\overa(s)$. 
The operator $A$ is injective since its non-diagonal part $J$ is invertible. The latter follows from the fact that, see~\eqref{Z0manifold},
\[
  \bigl\| I_{2N(N+1)}-J D_a F^{[N]}(\beta_0,\overa(0)) \bigr\|_{B(X^{[N]},X^{[N]})} \leq \max_{1\leq j\leq 4} Z_0^{(j)} < 1 ,
\] 
where the final inequality is implied by $p_j(r)<0$. Here the operator norm on $X^{[N]}\cong\R^{2N(N+1)}$ is induced by the one on $X=(\ell^1_\nu)^4$. Hence the fixed point $a(s)$ of $T$ solves $F(\beta_s,a(s))=0$.
By construction $Q_{\beta_s}$ is a parameterization of the local stable manifold defined for $\vert \theta\vert_{\infty} \leq \nu$, and for such $\theta$,
\begin{align*}
\left\vert Q_{\beta_s}(\theta)-\overline Q_{\beta_s}(\theta) \right\vert_{\infty} &= \left\vert \sum_{|\alpha|=0}^{\infty}  \left(a_{\alpha}(s) -\overa_{\alpha}(s)\right) \theta^{\alpha}\right\vert_{\infty}  \\
&= \max\limits_{j=1,\ldots,4}\left\vert \sum_{|\alpha|=0}^{\infty}  \left(a_{\alpha}^{(j)}(s) -\overa_{\alpha}^{(j)}(s)\right) \theta^{\alpha}\right\vert \\
&\leq \max\limits_{j=1,\ldots,4}\sum_{|\alpha|=0}^{\infty}  \left\vert a_{\alpha}^{(j)}(s) -\overa_{\alpha}^{(j)}(s)\right\vert \nu^{\vert\alpha\vert} \\
&= \max\limits_{j=1,\ldots,4}\left\Vert a^{(j)}(s) - \overa^{(j)}(s)\right\Vert_{1,\nu} \\
&\leq r. \qedhere
\end{align*}
\end{proof}
In the following section we use these approximations to rigorously prove the existence of homoclinic orbits for every parameter $\beta$ in $[0.5,1.9]$. To do so, we will also need control on the derivative of the parameterization $P_{\beta_s}$, which is provided by the theory of analytic functions. Define
\begin{equation}\label{hbeta}
h_{\beta_s}(\theta) \bydef P_{\beta_s}(\theta)-\overline P_{\beta_s}(\theta), \quad \theta\in\mathbb{R}^2, \ \vert\theta\vert_2\leq \nu.
\end{equation}
For all $s\in[0,1]$, the function $\hat{h}_{\beta_s}$, defined by~\eqref{e:defherror}, is analytic. Since $h_{\beta_s}(\theta)=\hat h_{\beta_s} (\theta_1+ \rm{i} \theta_2,\theta_1- \rm{i} \theta_2)$, we can control the derivative of $h_{\beta_s}$ (on a smaller domain) by a bound on $\hat h_{\beta_s}$. This is the content of the following lemma, of which the  proof can be found in \cite{MirelessMischaikow}.
\begin{lemma}\label{bornederive}
Assume that $\hat h:D_{\infty,\nu}(\mathbb{C}^2)\subset\mathbb{C}^2 \to \mathbb{C}^4$ is analytic, where
\begin{equation*}
D_{\infty,\nu}(\mathbb{C}^2) \bydef \left\{ \theta \in \mathbb{C}^2,\ \vert\theta\vert_{\infty}\leq \nu\right\},
\end{equation*}
and $\delta>0$ is such that 
\begin{equation} \label{eq:error_bounds_manifold}
\max_{\theta\in D_{\infty,\nu}(\mathbb{C}^2)} \left\vert  \hat h(\theta) \right\vert_{\infty} \leq \delta.
\end{equation}
Consider $h:D_{2,\nu}(\mathbb{R}^2)\subset\mathbb{R}^2 \to \mathbb{R}^4$ defined by $h(\theta)=\hat h (\theta_1+ \rm{i} \theta_2,\theta_1-  \rm{i} \theta_2)$, where
\begin{equation*}
D_{2,\nu}(\mathbb{R}^2)\bydef \left\{ \theta \in \mathbb{R}^2,\ \vert\theta\vert_{2}\leq \nu\right\}.
\end{equation*}
Then for any $\rho<\nu$ we have
\begin{equation} \label{eq:cauchy_bounds_derivatives}
\max\limits_{\theta\in D_{2,\rho}(\mathbb{R}^2)} \left\vert \frac{\partial h^{(j)}}{\partial \theta_i}(\theta) \right\vert_{\infty} \leq \frac{4\pi}{\nu \ln(\frac{\nu}{\rho})}\delta  \qquad \text{for } j=1,\ldots,4, ~i=1,2.
\end{equation}
\end{lemma}


\section{Parameterized families of symmetric homoclinic orbits}\label{orbit}

In this section, we apply the technique of Section~\ref{sec:radii_poly_approach} in a Chebyshev series setting to rigorously prove existence of parameterized families of symmetric homoclinic orbits. More precisely, we present all necessary estimates and bounds in order to demonstrate that solutions of \eqref{eq:brigdeODE}~exist for all $\beta \in [0.5,1.9]$. 

\subsection{A projected boundary value problem formulation}
\label{s:rewriting}

We begin by transforming the symmetric homoclinic orbit problem \eqref{eq:brigdeODE} into a projected boundary value problem (BVP). In order to set up the projected BVP, we first use the symmetry of the orbit to simplify the problem and therefore solve only for ``half of the orbit". The following lemma provides a strategy to do this.

\begin{lemma}
Let $u_0,u_2$ and $t_0$ be arbitrary numbers, and let $u(t)$ be the solution of the initial value problem

\[
\begin{cases} 
u''''(t) +\beta u''(t) +e^{u(t)}-1 = 0, \\
\left(u(t_0),u'(t_0),u''(t_0),u'''(t_0) \right) = \left(u_0, 0, u_2, 0 \right).
\end{cases}
\]
Then $u(-t+2t_0)=u(t)$ for all $t$ for which the solution $u$ is defined.
\end{lemma}

\begin{proof}
It is straightforward to verify that $u(-t+2t_0)$ is also a solution of the initial value problem. By the theorem of existence and uniqueness for ODEs, it follows that 
$u(-t+2t_0)=u(t)$ for all $t$ in the domain definition of $u$.
\end{proof}

Using the previous result, we fix a number $t_0=L>0$, and it follows that to solve \eqref{eq:brigdeODE}, it is enough to solve 
\begin{equation} \label{eq:BVP_at_infinity}
\begin{cases} u''''(t) +\beta u''(t) +e^{u(t)}-1 = 0, \\
u'(-L)=0, \quad u'''(-L) =0,  \\
\lim_{t \to \infty} \left( u(t),u'(t),u''(t),u'''(t) \right) = 0.
\end{cases}
\end{equation}
The idea now is to modify the boundary value problem \eqref{eq:BVP_at_infinity} in a way that the boundary value at $t=\infty$ is removed by a projected boundary value at $t=L$ where we impose at that time that $ \left( u(L),u'(L),u''(L),u'''(L) \right) \in W_{\rm loc}^s(0)$, a local stable manifold at $0$. In order to achieve this step, we use the theory of Section~\ref{Manifold} to obtain a real-valued parameterization $P_\beta$ of $W_{\rm loc}^s(0)$ at the parameter value $\beta \in [0.5,1.9]$:
\[
P_{\beta}(\theta)=Q_{\beta}(\theta_1+{\rm i} \theta_2,\theta_1- {\rm i} \theta_2)=
\sum_{|\alpha|=0}^{\infty} 
a_{\alpha}(\beta) (\theta_1+\text{i}\theta_2)^{\alpha_1} (\theta_1-\text{i}\theta_2)^{\alpha_2},
\]
which is well-defined for all $\theta \in D_{2,\tilde \nu}(\mathbb{R}^2) = \left\{ \theta \in \mathbb{R}^2 : \vert\theta\vert_{2} = \sqrt{\theta_1^2+\theta_2^2} \leq \tilde \nu \right\}$, where the size $\tilde \nu = \tilde \nu(\beta)$ of the domain of $P_\beta$ changes as the parameter $\beta \in [0.5,1.9]$ varies. Using the parameterization, we impose that \begin{equation}\label{e:Ltheta}
   (u(L),u'(L),u''(L),u'''(L))^T = {P}_\beta(\theta)
\end{equation}
for some $\theta \in D_{2,\tilde \nu}(\mathbb{R}^2)$, which implies that the orbit lies in the stable manifold. This introduces an indeterminacy that needs to be resolved. 
Namely, there is a one parameter family of pairs $(L,\theta)$ solving~\eqref{e:Ltheta} while describing the same orbit.
To overcome this, we impose that $\theta \in \partial D_{2,\rho}(\mathbb{R}^2) = \left\{ \theta \in \mathbb{R}^2 : \vert\theta\vert_{2} = \sqrt{\theta_1^2+\theta_2^2} = \rho \right\}$, for some fixed $\rho < \tilde{\nu}$, and we solve for the angle $\psi$. More precisely, we consider $\theta$ such that $\sqrt{\theta_1^2 + \theta_2^2} = \rho$ by setting $\theta_1 +\text{i}\theta_2 = \rho e^{\text{i} \psi}$ for some $\psi \in [0,2\pi)$. In this case, the evaluation of the parameterization of the local stable manifold along $\partial D_{2,\rho}(\mathbb{R}^2)$ reduces to
\begin{align*}
{P}_\beta(\psi) &= \sum_{|\alpha|=0}^\infty {a}_{\alpha}(\beta) (\theta_1+\text{i}\theta_2)^{\alpha_1} (\theta_1-\text{i}\theta_2)^{\alpha_2} \\
&= \sum_{|\alpha|=0}^\infty {a}_{\alpha}(\beta) \rho^{\alpha_1} e^{\text{i}\alpha_1\psi} \rho^{\alpha_2} e^{-\text{i}\alpha_2 \psi} \\
 &= \sum_{|\alpha|=0}^\infty {a}_{\alpha}(\beta) \rho^{|\alpha|} e^{\text{i}(\alpha_1-\alpha_2)\psi}.
\end{align*}
We slightly abuse notation by using the same notation ${P}_\beta$ to denote both ${P}_\beta(\theta)$ and ${P}_\beta(\psi)$. We can therefore define the projected BVP
\begin{equation} \label{eq:projected_BVP_1}
\begin{cases} u''''(t) +\beta u''(t) +e^{u(t)}-1 = 0, \quad t \in [-L,L], \\
u'(-L)=0, \quad u'''(-L) =0,  \\
(u(L),u'(L),u''(L),u'''(L))^T = {P}_\beta(\psi),
\end{cases}
\end{equation}
where $L>0$ and $\psi \in [0,2 \pi)$ are variables. As in Section~\ref{sec:introduction}, we make the change of variables 
\[
(v^{(1)},v^{(2)},v^{(3)},v^{(4)}) \bydef (e^{u_1}-1,u_2,u_3,u_4)
\]
and set $v=(v^{(1)},v^{(2)},v^{(3)},v^{(4)})$ to obtain that $v' = \Psi_{\beta}(v)$, where $\Psi_{\beta}:\R^4 \to \R^4$ is the vector field given by the right-hand side of \eqref{quadratic_system}. 
We rescale time via $t \mapsto t/L$ so that \eqref{eq:projected_BVP_1} becomes
\begin{equation} \label{e:reducedproblem}
\begin{cases} 
\dot v = L \Psi_{\beta}(v), \quad t \in [-1,1], \\
v^{(2)}(-1)=0, \quad v^{(4)}(-1) =0,  \\
v(1) = {P}_\beta(\psi).
\end{cases}
\end{equation}
A triplet $(L,\psi,v)$ satisfying \eqref{e:reducedproblem} thus corresponds to a symmetric homoclinic solution of the suspension bridge equation.
The rest of this section  is dedicated to applying the technique of Section~\ref{sec:radii_poly_approach} in a Chebyshev series setting to rigorously prove existence of parameterized families of solutions of the projected BVP \eqref{e:reducedproblem} for all $\beta \in [0.5,1.9]$.  This begins by defining a zero finding problem $F=0$ whose solutions correspond to symmetric homoclinic solutions of the suspension bridge equation.

\subsection{Setting up the zero finding problem using Chebyshev series}
\label{s:setupcheb}

Now that $v^{(i)}(t)$ is defined on $[-1,1]$ and needs to solve a boundary value problem, describing $v^{(i)}(t)$ in terms of a Chebyshev series is a natural choice,
see~\cite{CL,LR,BDLM,RayJB}.
Denote by $T_k:[-1,1] \to \R$ the $k$-th Chebyshev polynomial with $k \ge 0$, where $T_0(t)=1$, $T_1(t)=t$ and 
$T_{k+1}(t)=2t T_k(t)-T_{k-1}(t)$ for $k \ge 1$.
One way to characterize the Chebyshev polynomials is through the identity
$T_k(t) = \cos(k \arccos t)$, from which it follows that $\|T_k\|_\infty=1$,
$T_k(1)=1$, and $T_k(-1)=(-1)^k$.

For each $i=1,2,3,4$, we expand $v^{(i)}$ using a Chebyshev series expansion, that is
\begin{equation}\label{eq:ChebyExpansion}
v^{(i)}(t)= x_0^{(i)}+ 2 \sum_{k=1}^{\infty} x_k^{(i)} T_k(t).
\end{equation}
For each $i=1,2,3,4$, denote by $x^{(i)} \bydef \lbrace x_k^{(i)} \rbrace_{k \ge 0}$ the infinite dimensional vector of Chebyshev coefficients of $v^{(i)}$. 
The vector field is analytic (polynomial) and therefore the solutions (if they exist) of the projected BVP \eqref{e:reducedproblem} are analytic. By the Paley-Wiener theorem, this implies that the Chebyshev coefficients of each component of $v$ decay geometrically to zero. Hence, there exists a number $\nu>1$ such that $x^{(i)} \in \ell_\nu^1$ for each $i=1,2,3,4$, where 
\[
\ell_\nu^1 = \left\{ a = (a_k)_{k \ge 0} :  \left\| a \right\|_{1,\nu} \bydef |a_0|+ 2\sum_{k=1}^\infty |a_k| \nu^{k} < \infty
\right\}.
\]
We remark that throughout this section $\nu \geq 1$.

\begin{remark}[\bf Notation]
The decay rate $\nu$ in the definition of the Banach space $\ellnu$ appears both in the current section and in Section~\ref{Manifold}. Both values need not to be the same. Therefore, to avoid confusion, we denote by $\tilde{\nu}$ the value from Section~\ref{Manifold}. Moreover, although the sequence space $\ell_\nu^1$ as considered above is slightly different from the one used in Section~\ref{Manifold}, we nevertheless use the same notation, since the spaces and norms are completely analogous to those used in Section~\ref{s:fixedpointmanifold}. 
\end{remark}

The dual space can be characterized as follows. 
\begin{lemma}\label{l:dualbound}
The dual space $(\ell^1_\nu)^*$ is isomorphic to 
\[
\ell_{\nu^{-1}}^\infty = \left\{ c = (c_k)_{k \ge 0} : 
\left\| c \right\|_{\infty,\nu^{-1}} \bydef \max \left( |c_0|, \tfrac{1}{2} \sup_{k \geq 1} |c_k| \nu^{-k} \right) < \infty \right\}.
\]
For all $a \in \ell^1_\nu$ and $c \in \ell^\infty_{\nu^{-1}}$ we  have
\begin{equation}\label{e:dualbound}
\Bigl|\sum_{k\geq 0} c_k a_k \Bigr| \leq \|c\|_{\infty,\nu^{-1}} \|a\|_{1,\nu}.
\end{equation}
\end{lemma}
The following lemma is analogous to Lemma~\ref{lem:operator_norm}.
\begin{lemma}\label{l:Blnu1norm}
Let $\Gamma \in B(\ell^1_{\nu})$, the space of bounded linear operators from $\ell^1_\nu$ to itself, acting as $(\Gamma a)_i =\sum_{j\geq 0} \Gamma_{i,j} a_j$. Define the weights
$\omega=(\omega_k)_{k\geq0}$ by $\omega_0=1$ and $\omega_k = 2 \nu^k$ for $k\geq 1$. Then 
\[
   \| \Gamma \|_{B(\ell^1_\nu)} = \sup_{j \geq 0} \frac{1}{\omega_j}  \sum_{i\geq 0} | \Gamma_{i,j} | \omega_i .
\] 
\end{lemma}

The Banach space of unknowns $x \bydef (L,\psi,x^{(1)},x^{(2)},x^{(3)},x^{(4)})$ is
\begin{equation} \label{eq:Banach_Space_Chebyshev}
X \bydef \mathbb{R}^2 \times (\ellnu)^4 ,
\end{equation}
endowed with the norm 
\[
\|x\|_X \bydef \max \left\lbrace |L|,|\psi|,\|x^{(1)}\|_{1,\nu},\|x^{(2)}\|_{1,\nu},\|x^{(3)}\|_{1,\nu},\|x^{(4)}\|_{1,\nu} \right\rbrace.
\]

In terms of Chebyshev coefficients the differential equation $\dot{v}=L\Psi_{\beta}(v)$ becomes (see e.g.~\cite{LR}) 
\begin{equation}\label{e:fkj}
	\left\{
\begin{array}{rll}
f_k^{(1)}(\beta,x) &\!\!\!\!\bydef 2kx^{(1)}_k + L[x_{k\pm1}^{(2)}  +(x^{(1)}*x^{(2)})_{k\pm1} ]  &\!\!\!\!=0 , \\
f_k^{(2)}(\beta,x) &\!\!\!\!\bydef 2kx^{(2)}_k + Lx_{k\pm1}^{(3)} &\!\!\!\!=0 , \\
f_k^{(3)}(\beta,x) &\!\!\!\!\bydef 2kx^{(3)}_k + Lx_{k\pm1}^{(4)} &\!\!\!\!=0 ,  \\
f_k^{(4)}(\beta,x) &\!\!\!\!\bydef 2kx^{(4)}_k + L[-x_{k\pm1}^{(1)}- \beta x_{k\pm1}^{(3)}]  &\!\!\!\!=0 ,
\end{array} \right.
\end{equation}
for all $k \geq 1$. 
Here $x_{k\pm1}^{(i)} \bydef x_{k+1}^{(i)}-x_{k-1}^{(i)}$, 
and $*$ denotes the discrete convolution product $\ast: \ellnu \times \ellnu \to \ellnu$ defined as follows. Let $a,b \in \ellnu$, then for all $k\geq 0$ the $k$-th entry of the convolution product $a\ast b$ is given by
\[
(a\ast b)_k = \sum_{\substack{k_1+k_2=k \\ k_1,k_2 \in \mathbb{Z}}} a_{|k_1|} b_{|k_2|} .
\]
The choice of norm and convolution product is justified by the fact $\ellnu$ is a Banach algebra, that 
is $\left\| a \ast b \right\|_{1,\nu} \leq \left\| a \right\|_{1,\nu} \left\| b \right\|_{1,\nu} $, for all $a,b \in \ellnu$.

 The symmetry conditions $v^{(2)}(-1)=v^{(4)}(-1)=0$ reduce to
\begin{alignat}{2}
\eta^{(1)}(\beta,x) &\bydef \displaystyle x_0^{(2)}+ 2 \sum_{k=1}^{\infty} x^{(2)}_k(-1)^k
&& =0 ,  \label{eta1} \\
\eta^{(2)}(\beta,x) &\bydef \displaystyle x_0^{(4)}+ 2 \sum_{k=1}^{\infty} x^{(4)}_k(-1)^k
&& =0 , \label{eta2}
\end{alignat}
and the boundary conditions $v(1)=P_\beta(\psi)$ become
\begin{equation}
f_0^{(i)}(\beta,x) \bydef \displaystyle x_0^{(i)} +2\sum_{k=1}^{\infty} x^{(i)}_k - P_\beta^{(i)}(\psi) =0     \qquad \text{for } i=1,2,3,4.
\end{equation}
The full set of equations that we want to solve is thus $F(\beta,x)=0$,
where  
\begin{equation} \label{eq:definition_F}
F \bydef \left( \eta^{(1)},\eta^{(2)},F^{(1)},F^{(2)},F^{(3)},F^{(4)}\right),
\qquad 
\text{with } F^{(i)} \bydef  \left\lbrace f_k^{(i)}\right\rbrace_{k\geq 0}.
\end{equation}
In order to solve rigorously the problem $F(\beta,x)=0$ in the Banach space $X$, for all $\beta \in [0.5,1.9]$, we apply the radii polynomial approach of Section~\ref{sec:radii_poly_approach}.

\subsection{The finite dimensional reduction of the zero finding problem}

Having identified the operator $F$ given in \eqref{eq:definition_F} whose zeros correspond to symmetric homoclinic orbits of \eqref{eq:ode}, the next step is to compute numerical approximations, which requires considering a finite dimensional projection of the Banach space $X$ given in \eqref{eq:Banach_Space_Chebyshev}.
Given a sequence $a= (a_k)_{k \ge 0} \in \ell_\nu^1$, denote by $a^{[m]} = (a_0,\dots,a_{m-1}) \in \R^m$ the Galerkin projection 
of $a$ onto the first $m$ Chebyshev coefficients. Given an infinite dimensional vector $x=(L,\psi,x^{(1)},x^{(2)},x^{(3)},x^{(4)}) \in X$, 
denote 
\begin{equation} \label{eq:x^{[m]}}
x^{[m]} \bydef \bigl(L,\psi,(x^{(1)})^{[m]},(x^{(2)})^{[m]},(x^{(3)})^{[m]},(x^{(4)})^{[m]}\bigr) \in \R^2 \times (\R^m)^4 \cong \R^{4m+2}.
\end{equation} 
In this context, the finite dimensional Banach space $\R^{4m+2}$ is the finite dimensional projection of $X=\R^2 \times (\ell_\nu^1)^4$, and $x^{[m]}$ is the finite dimensional projection of $x$. 
We slightly abuse the notation by denoting $x^{[m]} \in X$ as the vector built from $x^{[m]} \in \R^{4m+2}$ by padding each entry $(x^{(i)})^{[m]}$ ($i=1,2,3,4$) with infinitely many zeros. 
The finite dimensional projection of $F$ given in \eqref{eq:definition_F} is defined as
\begin{align*}
F^{[m]} : \R \times \R^{4m+2} &\to \R^{4m+2}  \\
(\beta,x^{[m]}) & \mapsto F^{[m]} (\beta,x^{[m]}) \bydef  \bigl(F(\beta,x^{[m]}) \bigr)^{[m]}.
\end{align*}
We want to compute on $F^{[m]}$, but it depends on the parameterization $P_{\beta_s}$, which itself depends on \emph{infinitely} many Taylor coefficients. To remedy this, we consider a finite dimensional reduction of $P_{\beta_s}$. Recalling~\eqref{eq:barP_beta_s}, for every $s$, denote by $\overline{P}_{\beta_s}$ the computable approximation of the stable manifold given by
\begin{align}
    \overline{P}_{\beta_s}(\psi)&= \sum_{|\alpha|< N} (\overa_{0,\alpha}+s(\overa_{1,\alpha}-\overa_{0,\alpha}))\rho^{|\alpha|}e^{\text{i}\psi(\alpha_1-\alpha_2)} \nonumber \\
    &= \sum_{|\alpha|< N} (\overa_{0,\alpha}+s\Delta\overa_{\alpha})\rho^{|\alpha|}e^{\text{i}\psi(\alpha_1-\alpha_2)} \nonumber  \\
    &= \sum_{|\alpha|< N} \overa_{0,\alpha}\rho^{|\alpha|}e^{\text{i}\psi(\alpha_1-\alpha_2)} +s\sum_{|\alpha|< N} \Delta\overa_{\alpha}\rho^{|\alpha|}e^{\text{i}\psi(\alpha_1-\alpha_2)}
	 \nonumber  \\
    &= \overline{P}_{\beta_0}(\psi) + s \Delta \overline{P}(\psi),
	\label{e:defDeltaP}
\end{align}
where $\overa_{0,\alpha}$ and $\overa_{1,\alpha}$ are the numerical approximations of the coefficients of the stable manifold for $\beta_0$ and $\beta_1$ respectively. 

Finally, let $\overline{F}(\beta,x^{[m]})$ denote the finite dimensional projection of the operator using a Galerkin projection on the last four components and using the finite dimensional approximation $\overline{P}_{\beta}$ for the parameterization of the stable manifold. More explicitly, 
\begin{equation} \label{eq:definition_F_projection}
\overline{F}(\beta,x^{[m]}) \bydef \left( \overline{\eta}^{(1)}(\beta,x^{[m]}),\overline{\eta}^{(2)}(\beta,x^{[m]}),\overline{F}^{(1)}(\beta,x^{[m]}),\dots,\overline{F}^{(4)}(\beta,x^{[m]})\right),
\end{equation}
with $\overline{F}^{(i)}(\beta,x^{[m]}) \bydef \left\lbrace \overline{f}_k^{(i)}(\beta,x^{[m]}) \right\rbrace_{k = 0}^{m-1}$ for $i=1,2,3,4$, and 
\begin{align*} 
\overline{\eta}^{(1)}(\beta,x^{[m]}) &\bydef  x_0^{(2)}+ 2 \sum_{k=1}^{m-1} x^{(2)}_k(-1)^k, \qquad \overline{\eta}^{(2)}(\beta,x^{[m]}) \bydef  x_0^{(4)}+ 2 \sum_{k=1}^{m-1} x^{(4)}_k(-1)^k \\
\overline{f}_0^{(i)}(\beta,x^{[m]}) &\bydef x_0^{(i)} +2\sum_{k=1}^{m-1} x^{(i)}_k - \overline{P}_\beta^{(i)}(\psi) =0     \qquad \text{for } i=1,2,3,4,
\end{align*}
while $\overline{f}_k^{(i)}(\beta,x^{[m]})=f_k^{(i)}(\beta,x^{[m]})$ for all $k=1,\dots,m-1$, see~\eqref{e:fkj}.
Having identified $\overline{F}:\R \times \R^{4m+2} \to \R^{4m+2}:(\beta,x^{[m]}) \mapsto \overline{F}(\beta,x^{[m]})$ defined in \eqref{eq:definition_F_projection} as the finite dimensional reduction of $F$ given in \eqref{eq:definition_F}, we can apply the finite dimensional Newton's method to find numerical approximations. The next step is to define an infinite dimensional Newton-like operator $T:\R \times X \to X$ on which we apply the uniform contraction principle (via the radii polynomial approach of Section~\ref{sec:radii_poly_approach}).

\subsection{The Newton-like operator for the homoclinic orbit}

Let $\beta_0<\beta_1$ be two different parameter values, and consider two numerical approximations $\bx_0$ and $\bx_1$ such that $F(\beta_0,\bx_0)\approx 0$ and $F(\beta_1,\bx_1)\approx 0$. In practice we find $\bx_i$ by solving $\overline{F}(\beta_i,\cdot)=0$ numerically. For every $s\in [0,1]$, set
\[
    \bx_s = \bx_0 + s \Delta \bx,      \qquad \Delta \bx \bydef \bx_1-\bx_0
\]
and
\[
    \beta_s = \beta_0 + s\Delta \beta, \qquad \Delta \beta \bydef \beta_1-\beta_0.
\]
We denote  $\bx_s=(\overL_{s},\overp_{s},\bx_s^{(1)},\bx_s^{(2)},\bx_s^{(3)},\bx_s^{(4)}) \in X$ for 
$s \in [0,1]$,
and we recall that each $\bx_s^{(j)}$ is obtained from 
$(\bx_s^{(j)})^{[m]} \in \R^m$ by padding with zeros.  
Similarly, we denote
$\Delta \bx =
(\Delta\overL,\Delta\overp,\Delta\bx^{(1)},\Delta\bx^{(2)},\Delta\bx^{(3)},\Delta\bx^{(4)})$.

We now construct a fixed point operator $T(\beta,x) = x - A F(\beta,x)$ so that it is a uniform contraction over the 
interval of parameters $[\beta_0,\beta_1]$, whose fixed points $x=x(\beta)$ correspond to zeros of $F(\beta,\cdot)$ at a given parameter value $\beta \in [\beta_0,\beta_1]$. The operator $A$ is constructed as an approximate inverse of $DF(\beta_0,\bx_0)$.
Let $\bx_0$ be such that $\overline{F}(\beta_0,\bx_0)\approx 0$ and let $A^{[m]}\approx (D\overline{F}(\beta_0,\bx_0))^{-1}$ be a numerical approximation of the inverse of the Jacobian matrix. We decompose the $(4m+2) \times (4m+2)$ matrix $A^{[m]} $, into 36 blocks as
\begin{equation}\label{e:blocks}
  A^{[m]} = \left( \begin{array}{ccccc}
  A^{[m]}_{1,1} & A^{[m]}_{1,2} & A^{[m]}_{1,3} & \cdots & A^{[m]}_{1,6} \\
  A^{[m]}_{2,1} & A^{[m]}_{2,2} & A^{[m]}_{2,3} & \cdots & A^{[m]}_{2,6} \\
  A^{[m]}_{3,1} & A^{[m]}_{3,2} & A^{[m]}_{3,3} & \cdots & A^{[m]}_{3,6} \\
  \vdots & \vdots & \vdots & \ddots & \vdots \\
  A^{[m]}_{6,1} & A^{[m]}_{6,2} & A^{[m]}_{6,3} & \cdots & A^{[m]}_{6,6} 
  \end{array}
  \right) \, .
\end{equation}
Here $A^{[m]}_{i,j}$ is scalar for $1 \leq i,j \leq 2 $, 
$A^{[m]}_{i,j}$ is a row vector of length $m$ for $1\leq i \leq 2$, $3 \leq j\leq 6$, 
$A^{[m]}_{i,j}$ is a column vector of length $m$ for $3\leq i \leq 6$, $1 \leq j\leq 2$, and  
$A^{[m]}_{i,j}$ is a $m \times m$ matrix for $3\leq i,j \leq 6$.

\begin{definition}[\bf Definition of \boldmath$A$\unboldmath] \label{def:OperatorA}
 We  extend this finite dimensional operator $A^{[m]}= \lbrace A^{[m]}_{i,j} | 1\leq i,j \leq 6 \rbrace$ to an operator $A=\lbrace A_{i,j} | 1\leq i,j \leq 6 \rbrace$ on $X$ defined block-wise as
\begin{itemize}
	\item $A_{i,j} \in \mathbb{R}$ for $1 \leq i,j \leq 2$, where $A_{i,j}=A^{[m]}_{i,j}$;
	\item $A_{i,j} \in (\ellnu)^{*}$ for $1 \leq i \leq 2$ and $3\leq j \leq 6$, where $A_{i,j}$ is $A^{[m]}_{i,j}$ padded with zeros;
	\item $A_{i,j} \in \ellnu$ for $3 \leq i \leq 6$ and $1\leq j \leq 2$,
	where $A_{i,j}$ is $A^{[m]}_{i,j}$ padded with zeros;
	\item $A_{i,j} \in B(\ellnu,\ellnu)$ for $3\leq i,j \leq 6$, where
\begin{equation}\label{e:Adelta}
	(A_{i,j}x^{(j-2)})_k= \begin{cases}
	\left(A_{i,j}^{[m]}(x^{(j-2)})^{[m]}\right)_k &\quad \mbox{if}~ 0 \le k \leq m-1, \\[2mm]
	\frac{\delta_{i,j}}{2k}x_k^{(j-2)} & \quad\mbox{if}~ k\geq m ,
	\end{cases}
\end{equation}
with $\delta_{i,j}$ the usual Kronecker delta.
\end{itemize}
Here $(\ellnu)^{*}$ is the dual of $\ellnu$.
As an example, for $1 \leq i \leq 2$, $3\leq j \leq 6$, we have $A_{i,j} a = A^{[m]}_{i,j} a^{[m]}$.
The action of $A$ on $x=(L,\psi,x^{(1)},x^{(2)},x^{(3)},x^{(4)}) \in X$ is thus
\begin{alignat*}{1}
(Ax)^{(i)} & = A_{i,1}L +A_{i,2}\psi +\sum_{j=3}^6 A_{i,j}x^{(j-2)}, 
\qquad \text{for } 1\leq i \leq 6,
\end{alignat*}
where $(Ax)^{(i)} \in \mathbb{R}$ for $i=1,2$ and 
 $(Ax)^{(i)} \in \ell^1_{\nu}$ for $i=3,4,5,6$.
\end{definition}

We consider the Newton-like operator 
\begin{equation} \label{eq:newton-like_homoclinic}
  T(\beta_s,x)=x -AF(\beta_s,x)
\end{equation}
where $s \in[0,1]$ and $A$ is as in Definition~\ref{def:OperatorA}. 

\begin{lemma}
Given the operator $A$ as in Definition~\ref{def:OperatorA}. Then $T:\R \times X \to X$.
\end{lemma}

\begin{proof}
Consider $x \in X=\R^2 \times (\ell_\nu^1)^4$ and $\beta \in \R$.
By construction of $A$, in particular the infinite diagonal tail chosen in~\eqref{e:Adelta}, it is straightforward to verify that $\left( AF(\beta,x) \right)^{(i)} \in \mathbb{R}$ for $i=1,2$ and $\left( AF(\beta,x) \right)^{(i)} \in \ellnu$ for $3 \leq i \leq 6$. 
\end{proof}

Showing the existence of parameterized fixed points of $T$ defined in \eqref{eq:newton-like_homoclinic} is done by applying the general technique of Section~\ref{sec:radii_poly_approach}.
This requires computing the bounds $Y^{(j)}$ satisfying  \eqref{eq:Y_Bounds} and the bounds $Z^{(j)}$ satisfying \eqref{eq:Z_Bounds} for $j=1,\dots,6$. We recall that since 
$X = \prod_{j=1}^6 X_j = \mathbb{R}^2 \times (\ellnu)^4$, we have that $\|\cdot\|_{X_j}$ denotes the absolute value for $j=1,2$ and the $\ellnu$ norm for $j=3,4,5,6$.


\subsection{The \boldmath $Y$ \unboldmath bound for the homoclinic orbit problem}
We recall the definition of the bounds $Y^{(j)}$ in \eqref{eq:Y_Bounds}. In our context, $Y^{(j)}$ is a bound satisfying 
\[
	\sup_{s \in [0,1]} \| (A F(\beta_s,\bx_{s}) )^{(j)} \|_{X_j} \leq Y^{(j)} .
\]

We begin by expanding each component of 
\[
F(\beta_s,\bx_{s}) = \left( \eta^{(1)}(\beta_s,\bx_{s}),\eta^{(2)}(\beta_s,\bx_{s}),F^{(1)}(\beta_s,\bx_{s}),\dots,F^{(4)}(\beta_s,\bx_{s}) \right)
\]
as a polynomial in $s$. Given $s \in [0,1]$ and $j=1,2,3,4$, denote $\bx_{s}^{(j)} = \left( \bx_{s,k}^{(j)} \right)_{k \ge 0}$.

First, $\eta^{(1)}(\beta_s,\bx_s) = S_0^{(1)} +sS_1^{(1)}$ and $\eta^{(2)}(\beta_s,\bx_s) = S_0^{(2)} +sS_1^{(2)}$, where
\begin{align}
S_0^{(1)} & \bydef \bx_{0,0}^{(2)} +2\sum_{k=1}^{m-1} (-1)^k \bx_{0,k}^{(2)}, \qquad S_1^{(1)} \bydef \Delta\bx_0^{(2)} +2\sum_{k=1}^{m-1}(-1)^k\Delta \bx^{(2)}_k, \label{e:S1}
\\
S_0^{(2)} & \bydef \bx_{0,0}^{(4)} +2\sum_{k=1}^{m-1} (-1)^k\bx_{0,k}^{(4)}, \qquad S_1^{(2)} \bydef \Delta\bx_0^{(4)} +2\sum_{k=1}^{m-1}(-1)^k\Delta \bx_k^{(4)}. \label{e:S2}
\end{align}
Let us now expand $F^{(1)}(\beta_s,\bx_{s}),\dots,F^{(4)}(\beta_s,\bx_{s})$ as polynomials in $s$, and recall that their first component depend on the exact parameterization of the stable manifold $P_{\beta_s}$ which involves infinitely many Taylor coefficients. The work from Section~\ref{Manifold} provides the existence of a function $h_{s}:D_{2,\tilde \nu}(\mathbb{R}^2) \to \R^4$, see~\eqref{hbeta}, such that
\[
P_{\beta_s}(\theta) = \overline P_{\beta_s}(\theta) + h_{s}(\theta).
\]
As before, we slightly abuse notation by denoting $P_{\beta_s}(\psi) = \overline P_{\beta_s}(\psi) + h_{s}(\psi)$, where $\theta_1 +\text{i}\theta_2 = \rho e^{\text{i} \psi}$ for a fixed $\rho < \tilde{\nu}$.
We then split the operator as
\begin{equation*}
F(\beta_s,\bx_{s})=F^{(N)}(\beta_s,\bx_{s}) +H_s(\overline{\psi}_s),
\end{equation*}
where $F^{(N)}$ denotes the full infinite dimensional $F$ operator but evaluated using the (finitely computable) approximation of the parameterization $\overline{P}_{\beta_s}$ of order $N$, and where
\begin{align*}
H_s(\psi) \bydef \left(
\begin{array}{c}
0 \\
0 \\
( h_{s}^{(1)}(\psi),0,0,\hdots ) \\
( h_{s}^{(2)}(\psi),0,0,\hdots ) \\
( h_{s}^{(3)}(\psi),0,0,\hdots ) \\
( h_{s}^{(4)}(\psi),0,0,\hdots )
\end{array} \right).
\end{align*}
The size of $h_{s}^{(i)}(\psi)$ can be estimated by $r_m$ using Proposition~\ref{prop:validation_manifold}, where $r_m$ is the validation radius for the manifold for $\beta_0 \leq \beta \leq \beta_1$. In addition, since we know that the zeroth order term in $h_{s}^{(i)}$ vanishes, i.e.~$a_0(s)=\overa_0(s)=0$, we obtain a slightly sharper bound
for any $\rho <\tilde{\nu}$:
\begin{alignat*}{1}
   | h_{s}^{(i)}(\psi) | & = 
   \biggl|  \sum_{|\alpha|=0}^\infty \bigl(a^{(i)}_\alpha(s)-\overa^{(i)}_\alpha(s)\bigr) \rho^{|\alpha|}  e^{{\rm i}\psi (\alpha_1-\alpha_2)}  \biggr|
   \leq   \sum_{|\alpha|=1}^\infty \bigl|a^{(i)}_\alpha(s)-\overa^{(i)}_\alpha(s)\bigr| \left( \frac{\rho}{\tilde{\nu}}\right)^{|\alpha|} \tilde{\nu}^{|\alpha|} \\
 & \leq 
    \frac{\rho}{\tilde{\nu}}   \sum_{|\alpha|=1}^\infty \bigl|a^{(i)}_\alpha(s)-\overa^{(i)}_\alpha(s)\bigr|  \tilde{\nu}^{|\alpha|} 
	=
   \frac{\rho}{\tilde{\nu}} \bigl\|a^{(i)}(s)-\overa^{(i)}(s) \bigr\|_{1,\tilde{\nu}} \leq    \frac{\rho}{\tilde{\nu}}  r_m.
\end{alignat*}
Hence, we can estimate $H_s(\overp_s)$ elementwise by
\[
   |H_s(\overp_s)| \leq \mu \bydef   
   \left(
     \begin{array}{c}
       0 \\ 
       0 \\ 
       (\frac{\rho}{\tilde{\nu}} r_m,0,0,\hdots) \\ 
       (\frac{\rho}{\tilde{\nu}} r_m,0,0,\hdots) \\ 
	   (\frac{\rho}{\tilde{\nu}} r_m,0,0,\hdots) \\ 
	   (\frac{\rho}{\tilde{\nu}} r_m,0,0,\hdots) 
     \end{array}
   \right) .
\]

Denoting
\[
F^{(N)} = \left( \eta^{(1)},\eta^{(2)},F^{(1,N)},F^{(2,N)},F^{(3,N)},F^{(4,N)} \right)
\]
with $F^{(j,N)} = \left\{ f_0^{(j,N)},f_1^{(j)},f_2^{(j)},f_3^{(j)},\dots \right\}$ for $=1,2,3,4$, we rewrite $f_{0}^{(j,N)}(\beta_s,\bx_s)$ as a polynomial in $s$,
where we use $\Delta \overline{P}$ as defined in~\eqref{e:defDeltaP}:
\begin{align}
    f_{0}^{(j,N)}(\beta_s,\bx_s) &=  \bx_{0,0}^{(j)} +s\Delta\bx_0^{(j)} +2\sum_{k=1}^{m-1}\big[ \bx_{0,k}^{(j)} +s\Delta \bx_k^{(j)}\big] 
    -\overline{P}_{\beta_s}^{(j)}(\overp_s)  \nonumber  \\
    & = \left( \bx_{0,0}^{(j)} +2\sum_{k=1}^{m-1} \bx_{0,k}^{(j)} -\overline{P}_{\beta_0}^{(j)}(\overp_s)\right)
    +s\left(\Delta\bx_0^{(j)} +2\sum_{k=1}^{m-1} \Delta\bx_k^{(j)} -\Delta \overline{P}^{(j)}(\overp_s)\right)  \nonumber \\
    & = \left( \bx_{0,0}^{(j)} +2\sum_{k=1}^{m-1} \bx_{0,k}^{(j)} -\overline{P}_{\beta_0}^{(j)}(\overp_0)\right)  \nonumber  \\
    & \qquad +s\left(\Delta \bx_0^{(j)} +2\sum_{k=1}^{m-1} \Delta \bx_k^{(j)} - \Delta \overline{P}^{(j)}(\overp_0) - \Delta \overp \frac{d}{d \psi}\overline{P}_{\beta_0}^{(j)}(\xi)\right) - s^2 \Delta \overp \frac{d}{d \psi} \Delta \overline{P}^{(j)}(\zeta) 
	\nonumber \\
    & \bydef  S_{0,0}^{(j+2)} +s S_{1,0}^{(j+2)}+ s^2 S_{2,0}^{(j+2)},
	\label{e:defS0}
\end{align}
for some $\xi,\zeta $ between $\overp_0$ and $\overp_1$ (using the mean value theorem).
To obtain an explicit computable expression for $S_{1,0}^{(j+2)}$ and $S_{2,0}^{(j+2)}$, we determine 
\begin{alignat}{1}
  \frac{d}{d \psi}\overline{P}_{\beta_0}^{(j)}(\xi) 
  &=  
  \text{i} \sum_{|\alpha|< N} \overa_{0,\alpha} (\alpha_1-\alpha_2) \rho^{|\alpha|}e^{\text{i}\xi(\alpha_1-\alpha_2)}, \nonumber \\
\frac{d}{d \psi} \Delta \overline{P}^{(j)}(\zeta) 
  &=
 \text{i} \sum_{|\alpha|< N} \Delta \overa_{\alpha} (\alpha_1-\alpha_2) \rho^{|\alpha|}e^{\text{i}\zeta(\alpha_1-\alpha_2)}, \label{e:dDeltaPdpsi}
\end{alignat}
by an interval arithmetic calculation, i.e., replacing $\xi$ and $\zeta$
by the interval $[\overp_0,\overp_1]$.

For $k \ge 1$, we set ($j=1,2,3,4$)
\begin{equation}\label{e:fkexpandedins}
	f_{k}^{(j)}(\beta_s,\bx_s)= S_{0,k}^{(j+2)}+ S_{1,k}^{(j+2)} s +S_{2,k}^{(j+2)}s^2+S_{3,k}^{(j+2)}s^3,
\end{equation} 
where the third order term is nonzero for $j=1$ and $j=4$ only. All terms are collected in Table \ref{Table:CoeffS}. Then, it is possible to write the whole operator as
\begin{equation}\label{splitted}
F(\beta_s,\bx_{s})=S_0+ sS_1 +s^2S_2 +s^3 S_3 + H_s(\overline{\psi}_s), 
\end{equation}
where $S_i=( S_i^{(1)},S_i^{(2)}, \lbrace S_{i,k}^{(3)} \rbrace_{k\geq0}, \lbrace S_{i,k}^{(4)}\rbrace_{k\geq0}, \lbrace S_{i,k}^{(5)}\rbrace_{k\geq0}, \lbrace S_{i,k}^{(6)}\rbrace_{k\geq0} )$ for $i=0,1,2,3$. 

%
%
%
%
\begin{table}[ht!]\label{Table:CoeffS}
\centering
\begin{tabular}{|c|c|}
\hline
$i$ & Coefficients $S_{i,k}^{(3)}$ for $k \geq 1$ \\
\hline
$0$ & $2k \bx_{0,k}^{(1)} +\overL_0\left[\bx_{0,k\pm1}^{(2)} +(\bx_0^{(1)}*\bx_0^{(2)})_{k\pm1} \right]$ \\
$1$ & $2k \Delta \bx^{(1)}_k +\overL_0\left[ \Delta \bx^{(2)}_{k\pm1} +(\Delta \bx^{(1)}\ast\bx_0^{(2)})_{k\pm1} +(\bx_0^{(1)}\ast\Delta \bx^{(2)})_{k\pm1}\right]$ \\ 
 & $+\Delta\overL\left[\bx^{(2)}_{0,k\pm1} +(\bx_0^{(1)}\ast \bx_0^{(2)})_{k\pm1}\right]$ \\
$2$ & $\overL_0 (\Delta \bx^{(1)} \ast \Delta\bx^{(2)})_{k\pm1} + \Delta\overL\left[ \Delta\bx^{(2)}_{k\pm1}+(\Delta\bx^{(1)} \ast \bx_0^{(2)})_{k\pm1} +(\bx_0^{(1)}\ast\Delta\bx^{(2)})_{k\pm1}\right]$\\ 
$3$ & $\Delta \overL(\Delta\bx^{(1)} \ast \Delta\bx^{(2)})_{k\pm1}$ \\
\hline
$i$ & Coefficients $S_{i,k}^{(4)}$ for $k \geq 1$ \\
\hline
$0$ & $2k\bx_{0,k}^{(2)} +\overL_0 \bx^{(3)}_{0,k\pm1}$ \\
$1$ & $\overL_0\Delta\bx^{(3)}_{k\pm1} +2k\Delta\bx^{(2)}_k+ \Delta \overL \bx^{(3)}_{0,k\pm1}$ \\
$2$ & $\Delta \overL\Delta\bx^{(3)}_{k\pm1}$ \\
$3$ & $0$ \\
\hline
$i$ & Coefficients $S_{i,k}^{(5)}$ for $k \geq 1$ \\
\hline
$0$ & $2k\bx_{0,k}^{(3)} +\overL_0\bx^{(4)}_{0,k\pm1}$ \\
$1$ & $\overL_0\Delta\bx^{(4)}_{k\pm1} +2k\Delta\bx^{(3)}_k+ \Delta\overL\bx^{(4)}_{0,k\pm1}$ \\
$2$ & $\Delta \overL\Delta\bx^{(4)}_{k\pm1}$ \\
$3$ & $0$ \\
\hline
 $i$ & Coefficients $S_{i,k}^{(6)}$ for $k \geq 1$ \\
\hline
$0$ & $2k\bx_{0,k}^{(4)} - \overL_0\left[\bx^{(1)}_{0,k\pm 1} + \beta_0\bx^{(3)}_{0,k\pm1}\right]$ \\
$1$ & $2k\Delta\bx^{(4)}_k - \Delta \overL \left[\bx^{(1)}_{0,k\pm1}+ \beta_0 \bx^{(3)}_{0,k\pm1}\right] - \overL_0\left[\Delta \bx^{(1)}_{k\pm1}+\beta_0 \Delta\bx^{(3)}_{k\pm1} +\Delta\beta\bx^{(3)}_{0,k\pm1}\right]$ \\
$2$ & $- \Delta \overL \left[\Delta\bx^{(1)}_{k\pm1}+\beta_0 \Delta\bx^{(3)}_{k\pm1}+\Delta\beta\bx^{(3)}_{0,k\pm1}\right] -\overL_0\Delta\beta \Delta\bx^{(3)}_{k\pm1}$ \\ 
$3$ & $ - \Delta \overL \Delta\beta \Delta \bx^{(3)}_{k\pm1}$ \\
\hline
\end{tabular}
\caption{Coefficients $S_{i,k}^{(j)}$ for the splitting of $F(\beta_s,\bx_{s})$ as a polynomial in $s$, as given in \eqref{splitted}.
The coefficients $S_{i,k}^{(j)}$ for $k=0$ and $3 \leq j \leq 6$ are provided in~\eqref{e:defS0}.
The coefficients $S_{i,k}^{(j)}$ for $j=1,2$ are provided in~\eqref{e:S1} and~\eqref{e:S2}.}
\end{table}

Since we evaluate $F$ using a finite dimensional approximation, $F(\beta_s,\bx_s)$ will contain only a finite number of nonzero elements. First, we consider the entries not exceeding the dimension of the finite dimension approximation. This part gets `hit' by $A^{[m]}$,
and is bounded component by component:
\begin{align*}
\bigl|A^{[m]} F (\beta_s,\bx_{s})^{[m]} \bigr|  \leq V \bydef &
\bigl|A^{[m]}S_0^{[m]}\bigr| +
\bigl|A^{[m]}S_1^{[m]}\bigr| +
\bigl|A^{[m]}S_2^{[m]}\bigr| + 
\bigl|A^{[m]}S_3^{[m]}\bigr| +\bigl| A^{[m]} \bigr| \mu^{[m]}  ,
 \end{align*}  
where
$V=(V^{(1)},V^{(2)},\{V^{(3)}_k\}_{k=0}^{m-1},\{V^{(4)}_k\}_{k=0}^{m-1},\{V^{(5)}_k\}_{k=0}^{m-1},\{V^{(6)}_k \}_{k=0}^{m-1})^T$.
Concerning terms that exceed the dimension of the finite dimensional projection,
by using the definition of $A$, one gets for $j=3,4,5,6$
\[
  \left|\big(A_{j,j}F^{(j-2)}(\beta_s,\bx_s)\big)_k\right| = \left|\frac{1}{2k}f_{k}^{(j-2)}(\beta_s,\bx_s)\right|
  \qquad\text{for } k \geq m.
\]
The expansion of $f_{k}^{(j-2)}(\beta_s,\bx_s)$ in powers of $s$ is given by~\eqref{e:fkexpandedins} with the coefficients in Table~\ref{Table:CoeffS}, We note that all $S^{(j)}_{i,k}$ vanish for $k \geq 2m$. 
To be precise, $S^{(3)}_{i,k}$ vanishes for $k \geq 2m$, whereas  when $j=4,5,6$ then
$S^{(j)}_{i,k}$ vanishes already for $k \geq m+1$. Hence we define the estimates ($3 \leq j \leq 6$, $m\leq k\leq 2m-1$)
\[
  \frac{1}{2k} \left|f_{k}^{(j-2)}(\beta_s,\bx_s)\right| 
  \leq W^{(j)}_k \bydef \frac{1}{2k} \left(
   | S_{0,k}^{(j)} | + |S_{1,k}^{(j)}| +  
   |S_{2,k}^{(j)} | + |S_{3,k}^{(j)}| \right) .
\]

Having estimated all the terms appearing in the expression $\|(T(\beta_s,\bx_s)-\bx_s)^{(j)}\|_{X_j}$, we set
\begin{equation*} 
Y^{(j)} \bydef \begin{cases} V^{(j)}, &\qquad j=1,2, \\
\displaystyle V_0^{(j)} +2\sum_{k=1}^{m-1} V_k^{(j)}\nu^k +2\sum_{k=m}^{2m-1} W_k^{(j)} \nu^k, &\qquad j=3,4,5,6.      
\end{cases}
\end{equation*}
By construction, we have
\begin{align*}
\left\|(T(\beta_s,\bx_s) - \bx_s)^{(j)}\right\|_{X_j} \leq Y^{(j)} 
\qquad \text{for all } s\in [0,1] \text{ and }  j=1,\hdots,6.
\end{align*}
%


\subsection{The \boldmath $Z$ \unboldmath bound for the homoclinic orbit problem}

We recall the definition of the bounds $Z^{(j)}$ in \eqref{eq:Z_Bounds}. In our context, $Z^{(j)}$ is a bound satisfying 
\[
 \sup_{\substack{b,c \in B_r(0) \\ s \in[0,1]}} \left\| D_xT^{(j)}(\beta,\bx_s +b)c \right\|_{X_j}   \leq Z^{(j)}(r).
\]

To simplify the manipulations of the expressions appearing in the bounds, we introduce an operator $A^\dagger = \lbrace A^\dagger_{i,j} | 1\leq i,j\leq 6 \rbrace$, where the splitting is explained in Definition~\ref{def:OperatorA}.
This operator $A^\dagger$ is on the one hand an `almost inverse' of the operator $A$, and on the other hand it approximates $D_x F(\beta_0,\bx_0)$.
We define $A^\dagger$ piecewise, where we use the decomposition of the Jacobian
$(D_x\overline{F}(\beta_0,\bx_0)) =(D_x\overline{F}(\beta_0,\bx_0)) _{i,j}$ into 36~blocks as in~\eqref{e:blocks}:
\begin{itemize}
  \item $A_{i,j}^\dagger\in \mathbb{R}$ for $1 \leq i,j \leq 2$, where 
  $A_{i,j}^\dagger = (D_x\overline{F}(\beta_0,\bx_0))_{i,j}$;
  \item $A_{i,j}^\dagger \in  (\ellnu)^{*}$ for $1 \leq i \leq 2$ and $3 \leq j \leq 6$, where
 $A_{i,j}^\dagger$ is $(D_x\overline{F}(\beta_0,\bx_0))_{i,j}$ padded with zeros;
  \item $A_{i,j}^\dagger \in  \ellnu$ for $3 \leq i \leq 6$ and $1 \leq j \leq 2$, where $A_{i,j}^\dagger$ is $(D_x\overline{F}(\beta_0,\bx_0))_{i,j}$ padded with zeros;
  \item $A_{i,j}^\dagger \in B(\ellnu,\ell_{\nu'}^1)$ for $3 \leq i,j \leq 6$, with $\nu' < \nu$, where
\[
	(A_{i,j}^\dagger x^{(j-2)})_k= \begin{cases} 
	\Bigl((D_x\overline{F}(\beta_0,\bx_0))_{i,j}(x^{(j-2)})^{[m]}\Bigr)_k & \quad \mbox{if}~ 0 \le k \leq m-1, \\[2mm]
	\delta_{i,j} 2k x_k^{(j-2)} &\quad \mbox{if}~ k\geq m .
	\end{cases}
\]
\end{itemize}
Now, we use $A^\dagger$ to perform the splitting
\begin{align}
\nonumber DT(\beta_s,\bx_s+b)c &= [I-ADF(\beta_s,\bx_s +b)]c \\ 
&= [I-AA^\dagger]c - A[DF(\beta_s,\bx_s +b)c -A^\dagger c]. \label{separation2}
\end{align}
As in Section~\ref{Manifold}, the bound on the first term in \eqref{separation2} can be directly computed. We set $B=I-AA^\dagger$, whose nonzero elements are represented by the finite matrix
$I_{4m+2}-A^{[m]}D_x\overline{F}(\beta_0,\bx_0)$, and we use Lemmas~\ref{l:dualbound} and~\ref{l:Blnu1norm} to derive the bounds 
\begin{equation}\label{e:Z0}
Z_0^{(i)}\bydef 
\begin{cases}
\displaystyle \sum_{j=1}^2 |B_{i,j}| 
+ \sum_{j=3}^6  \| B_{i,j} \|_{\infty,\nu^{-1}}  &\quad\text{for } i=1,2, \\
\displaystyle \sum_{j=1}^2 \|B_{i,j}\|_{1,\nu} 
+ \sum_{j=3}^6  \| B_{i,j} \|_{B(\ell^1_\nu)} &\quad\text{for }  i=3,4,5,6,
\end{cases}
\end{equation}
with the norms introduced in Section~\ref{s:setupcheb}.
This provides the desired bound on the first term of \eqref{separation2}. For the second term, we set $u,v \in B_1(0)$ such that $b=ru$ and $c=rv$. We denote 
$v=(v_{L},v_{\psi},v^{(1)},v^{(2)},v^{(3)},v^{(4)})$, and similarly for $u$, $b$ and $c$. First, for $i=1,2$, we have
\begin{align}\label{Boundz}
   \nonumber [DF(\beta_s,\bx_s +b)c-A^\dagger c]^{(i)} &= \left| c^{(2i)}_0 +2\displaystyle\sum_{k=1}^\infty (-1)^k c^{(2i)}_k -c^{(2i)}_0 -2 \sum_{k=1}^{m-1} (-1)^k c^{(2i)}_k  \right| \\
&= \left| \displaystyle  2\sum_{k=m}^{\infty} (-1)^k c^{(2i)}_k \right|\leq 2\displaystyle\sum_{k=m}^{\infty}|v^{(2i)}_k|r\leq\frac{1}{\nu^m}r,
\end{align}
where the final inequality follows from Lemma~\ref{l:dualbound}. 
Next we consider the $k=0$ term of the other four components.
For $i=1,2,3,4$ one finds
\begin{align}
 \left|[DF(\beta_s,\bx_s +b)c-A^\dagger c]^{(i+2)}_0 \right| \hspace*{-2.5cm} & \nonumber\\
      =\, &\left| \bigg[ \displaystyle-\frac{dP_{\beta_s}^{(i)}}{d \psi}  (\overline{\psi}_s+b_\psi) 
	 c_\psi + c^{(i)}_0+ \displaystyle 2\sum_{k=1}^\infty c^{(i)}_k \bigg]
	 -\bigg[ -\frac{d \overline{P}_{\beta_0}^{(i)}}{d \psi}(\overline{\psi}_0) 
c_\psi +c^{(i)}_0 + 2\displaystyle\sum_{k=1}^{m-1} c^{(i)}_k \bigg] \right| \nonumber\\
    \leq\, & \left| \frac{d \overline{P}_{\beta_0}^{(i)}}{d\psi}(\overline{\psi}_s+r u_\psi) 
	-\frac{d\overline{P}_{\beta_0}^{(i)}}{d\psi}(\overline{\psi}_s) 
	+\frac{d\overline{P}_{\beta_0}^{(i)}}{d\psi}(\overline{\psi}_s) 
-\frac{d\overline{P}_{\beta_0}^{(i)}}{d\psi}(\overline{\psi}_0)
	+s\frac{d\Delta \overline{P}^{(i)}}{d\psi}(\overline{\psi}_s+b_\psi) 
	\right| r \nonumber\\
	 &\quad+ \left( \left| \frac{d h_s}{d \psi}(\overline{\psi}_s+b_\psi) \right| + \left| 2\displaystyle\sum_{k=m}^{\infty} v^{(i)}_k \right|
	  \right)r \nonumber \\
    \leq\, &
	   \left| \frac{d^2 \overline{P}_{\beta_0}^{(i)}}{d \psi^2}(\zeta_s) \right| r^2
	  +	\left( \left| \frac{d^2 \overline{P}_{\beta_0}^{(i)}}{d \psi^2}(\xi_s) \Delta \overp \right| 
	  +\left|\frac{d \Delta \overline{P}^{(i)}}{d \psi}(\overline{\psi}_s+b_\psi)  \right| \right) s r 
	  \label{eq:TermForMVT}\\
	 &\quad+ \left( \rho\left| \frac{\partial h_s}{\partial \theta_1}(\overline{\psi}_s+b_\psi) \right| +\rho\left| \frac{\partial h_s} {\partial \theta_2}(\overline{\psi}_s+b_\psi)\right| 
	  +\frac{1}{\nu^m} \right)r, \label{eq:dhsdtheta}
\end{align}
where $\zeta_s$ is in $[\overp_s-r,\overp_s+r]$, and $\xi_s$ is in $[\overp_0,\overp_s]$.
A direct computation shows that
\begin{align*}
  \left| \frac{d^2 \overline{P}_{\beta_0}^{(i)}}{d \psi^2}(\psi) \right| &
= \left| \sum_{|\alpha|<N} -\overa_{0,\alpha}^{(i)} \rho^{|\alpha|} (\alpha_1-\alpha_2)^2 e^{\text{i}\psi(\alpha_1-\alpha_2)} \right| 
  \leq \sum_{|\alpha|< N} \left|\overa_{0,\alpha}^{(i)}\right| \rho^{|\alpha|} (\alpha_1-\alpha_2)^2 . 
\end{align*}
Combining this with~\eqref{e:dDeltaPdpsi} gives us a bound on the terms in~\eqref{eq:TermForMVT}: 
\begin{alignat*}{1}
 \Lambda^{(i)} &\bydef |\Delta \overp| \sum_{|\alpha|< N} \left|\overa_{0,\alpha}^{(i)} \right| \rho^{|\alpha|} (\alpha_1-\alpha_2)^2  + \sum_{|\alpha|< N} \left|\Delta \overa_{0,\alpha}^{(i)} \right| \rho^{|\alpha|} |\alpha_1-\alpha_2| , \\
 \tilde{\Lambda}^{(i)} &\bydef \sum_{|\alpha|< N} \left|\overa_{0,\alpha}^{(i)} \right| \rho^{|\alpha|} (\alpha_1-\alpha_2)^2 .
\end{alignat*}
The remaining terms in~\eqref{eq:dhsdtheta} are estimated using Lemma~\ref{bornederive}.
We obtain, for $i=1,2,3,4$,
\begin{align}\label{Boundz0}
    \left|[DF(\beta_s,\bx_s +b)c-A^\dagger c]^{(i+2)}_0 \right| 
    \leq \mathcal{W}^{(i+2)}_1 r + \tilde{\Lambda}^{(i)} r^2  ,
\end{align}
with, for $j=3,4,5,6$,
\begin{equation}\label{e:W1}
  \mathcal{W}^{(j)}_1  \bydef \left(\Lambda^{(j-2)} + \frac{8\pi\rho r_m}{\tilde{\nu} \ln \frac{\tilde{\nu}}{\rho}}+\frac{1}{\nu^m} \right) .
\end{equation}

For $k\neq 0$, we consider separately the coefficients of $r$, $r^2$ and $r^3$: 
\begin{align*}
\left(DF(\beta_s,\bx_s +ru)rv-A^\dagger rv \right)_k^{(i)} = \tilde{z}_{1,k}^{(i)}r + \tilde{z}_{2,k}^{(i)}r^2 + \tilde{z}_{3,k}^{(i)}r^3,
\qquad\text{for } i=3,4,5,6.
\end{align*}
The term $ -A^\dagger v$ contributes to the $s$-independent part of
$\tilde{z}_{1,k}^{(i)}$ only. Since in $\tilde{z}_{1,k}^{(i)}$ some of the terms involving $(v^{(j)})^{[m]}$ will cancel, it is useful to introduce
$\hv^{(j)}$ as follows:
\[
  \hv^{(j)}_k \bydef \begin{cases} 0 & \text{if } k<m,\\
  v^{(j)}_k & \text{if } k \geq m. \end{cases}
\]
Using this notation, for $\tilde{z}_{1,k}^{(3)}$ and $1 \leq k \leq m-1$, one finds
%
%
%
%
\begin{align*}
    \tilde{z}_{1,k}^{(3)}&= 
    \overL_0\left[(\bx_0^{(1)}\ast \hv^{(2)})_{k\pm1} +(\hv^{(1)}\ast \bx_0^{(2)})_{k\pm1} + \delta_{k,m-1} \hv_{k\pm1}^{(2)}\right] \\
    &\qquad +s\bigg(v_L \left[ \Delta \bx^{(2)}_{k\pm1}  + (\bx^{(1)}_0\ast \Delta\bx^{(2)})_{k\pm1} + (\Delta\bx^{(1)} \ast \bx^{(2)}_0)_{k\pm1}\right]
	 \\
    &\qquad +\overL_0\left[ (\Delta \bx^{(1)} \ast v^{(2)})_{k\pm1} +(v^{(1)}\ast \Delta \bx^{(2)})_{k\pm1}\right] +\Delta\overL \left[  v_{k\pm1}^{(2)} + (\bx_0^{(1)}\ast v^{(2)})_{k\pm1} +(v^{(1)}\ast \bx_0^{(2)})_{k\pm1}\right]\bigg)\\
    &\qquad + s^2\left(\Delta\overL \big[(\Delta \bx^{(1)} \ast v^{(2)})_{k\pm1}+ (v^{(1)}\ast\Delta \bx^{(2)})_{k\pm1}\big]+v_L(\Delta \bx^{(1)}\ast \Delta \bx^{(2)})_{k\pm1} \right) . 
\end{align*}
Clearly $\delta_{k,m-1} \hv_{k\pm1}^{(2)} = \delta_{k,m-1} v_{k+1}^{(2)} =\hv_{k\pm1}^{(2)} $, for $k \leq m-1$, and
the Kronecker $\delta_{k,m-1}$ may be viewed as superfluous.
For $k \geq m$ we find
\begin{align*}
\tilde{z}_{1,k}^{(3)}&= 
    \overL_0\left[(\bx_0^{(1)}\ast v^{(2)})_{k\pm1}+(v^{(1)}\ast \bx_0^{(2)})_{k\pm1}+ v_{k\pm1}^{(2)}\right] +v_L\left[(\bx_0^{(1)}\ast \bx_0^{(2)})_{k\pm1} +\delta_{k,m}\bx^{(2)}_{0,k\pm1}\right]\\
    &\qquad +s\bigg(
	v_L \left[\delta_{k,m}\Delta \bx^{(2)}_{k\pm1}+ (\bx_0^{(1)}\ast\Delta\bx^{(2)})_{k\pm1} +(\Delta\bx^{(1)}\ast \bx_0^{(2)})_{k\pm1}\right]
 \\
    &\qquad +
	\overL_0 \left[(\Delta \bx^{(1)}\ast v^{(2)})_{k\pm1}+(v^{(1)}\ast\Delta\bx^{(2)})_{k\pm1}\right]+\Delta\overL\left[v_{k\pm1}^{(2)} +(\bx_0^{(1)}\ast v^{(2)})_{k\pm1}+(v^{(1)}\ast \bx_0^{(2)})_{k\pm1}\right]
	\bigg) \\
    &\qquad + s^2\bigg(\Delta\overL \left[(\Delta \bx^{(1)} \ast v^{(2)})_{k\pm1} +(v^{(1)} \ast\Delta\bx^{(2)})_{k\pm1}\right] +v_L(\Delta \bx^{(1)}\ast \Delta \bx^{(2)})_{k\pm1}\bigg).
\end{align*}
Once again the Kronecker $\delta_{k,m}$ may be viewed as superfluous.
%
%
%
For $\tilde{z}_{4,k}^{(1)}$, $\tilde{z}_{5,k}^{(1)}$ and $\tilde{z}_{6,k}^{(1)}$, one finds
\begin{alignat*}{1}
\tilde{z}_{1,k}^{(4)} & =\begin{cases}
\delta_{k,m-1} \overL_0 v^{(3)}_{k+1} +
  s\big[\Delta\overL v^{(3)}_{k\pm1} + v_L \Delta \bx^{(3)}_{k\pm1}\big] & \quad\text{for } 1\leq k\leq m-1 \\
  %
  \overL_0 v^{(3)}_{k\pm1}+ v_L \delta_{k,m}\bx^{(3)}_{0,k\pm1} +s\big[\Delta\overL v^{(3)}_{k\pm1} + \delta_{k,m} v_L \Delta\bx^{(3)}_{k\pm1}\big]& \quad\text{for } k \geq m,
\end{cases} 
\\[1mm]
\tilde{z}_{1,k}^{(5)}& =\begin{cases}
 \delta_{k,m-1} \overL_0 v^{(4)}_{k+1}+  s\big[\Delta\overL v^{(4)}_{k\pm1} + v_L \Delta \bx^{(4)}_{k\pm1}\big] & \quad\text{for } 1\leq k\leq m-1 \\
  \overL_0 v^{(4)}_{k\pm1} +v_L\delta_{k,m}\bx^{(4)}_{0,k\pm1} +s\big[\Delta\overL v^{(4)}_{k\pm1} +\delta_{k,m} v_L \Delta\bx^{(4)}_{k\pm1}\big]& \quad\text{for } k \geq m,
\end{cases}
\end{alignat*}
and 
\begin{equation*}
\tilde{z}_{1,k}^{(6)}=\begin{cases}
    - \delta_{k,m-1} \overL_0\big[ \beta_0 v^{(3)}_{k+1} +v^{(1)}_{k+1}\big]
    - s\bigg(v_L \big[\Delta\beta \bx^{(3)}_{0,k\pm1} +\Delta\bx^{(1)}_{k\pm1}+\beta_0\Delta\bx^{(3)}_{k\pm1}\big]    \\ 
    \qquad +\Delta\overL v^{(1)}_{k\pm1} + \big[\overL_0 \Delta\beta +\Delta \overL \beta_0\big] v^{(3)}_{k\pm1}\bigg) - s^2\big[
	\Delta\overL\Delta\beta v^{(3)}_{k\pm1} + v_L\Delta\beta \Delta \bx^{(3)}_{k\pm1}\big]& \quad\text{for } 1\leq k\leq m-1 \\[5mm]
    - \overL_0\big[ \beta_0 v^{(3)}_{k\pm1} +v^{(1)}_{k\pm1}\big]
    - \delta_{k,m}v_L\big[\beta_0 \bx^{(3)}_{0,k\pm1} +\bx^{(1)}_{0,k\pm1}\big] \\
    \qquad- s\bigg(\delta_{k,m}v_L\big[\Delta\beta\bx^{(3)}_{0,k\pm1} +\Delta\bx^{(1)}_{k\pm1}+ \beta_0 \Delta\bx^{(3)}_{k\pm1}\big] +\Delta\overL v^{(1)}_{k\pm1}  \\ 
    \qquad\qquad + \big[\overL_0 \Delta\beta +\Delta \overL \beta_0\big] v^{(3)}_{k\pm1}\bigg) - s^2\big[\Delta\overL \Delta\beta v^{(3)}_{k\pm1} +\delta_{k,m} v_L\Delta\beta\Delta\bx^{(3)}_{k\pm1}\big]& \quad\text{for } k \geq m.
\end{cases}
\end{equation*}
The $\tilde{z}_{i,k}^{(2)}$ and $\tilde{z}_{i,k}^{(3)}$ coefficients are still to be determined. For $k\neq 0$, they are given in Table~\ref{table:Coeff}.
\begin{table}[h]
\centering
\begin{tabular}{|c|l|}
\hline
\multicolumn{2}{|c|}{Coefficients in front of $r^2 $, for $k \geq 1$} \\
\hline
    $\tilde{z}_{2,k}^{(3)}$ 
    &
    $ v_L \big( (\bx^{(1)}_0\ast u^{(2)})_{k\pm1}+ (u^{(1)} \ast \bx^{(2)}_0)_{k\pm1} + s[(\Delta \bx^{(1)} \ast u^{(2)})_{k\pm1}+(u^{(1)}\ast\Delta \bx^{(2)})_{k\pm1}]+u^{(2)}_{k\pm1}\big)$\\ 
    & 
    $+ u_L \big( (\bx^{(1)}_0 \ast v^{(2)})_{k\pm1} +(v^{(1)} \ast \bx^{(2)}_0)_{k\pm1} +s[(\Delta \bx^{(1)} \ast v^{(2)})_{k\pm1}+(v^{(1)} \ast\Delta \bx^{(2)})_{k\pm1}] +v^{(2)}_{k\pm1}\big) $ \\
    &
    $+ s\Delta\overL \big( (u^{(1)} \ast v^{(2)})_{k\pm1}+ (v^{(1)} \ast u^{(2)})_{k\pm1}\big) +\overL_0\big((u^{(1)}\ast v^{(2)})_{k\pm1}+ (v^{(1)}\ast u^{(2)})_{k\pm1}\big) $ \\
\hline
    $\tilde{z}_{2,k}^{(4)}$ 
    &
    $ u_L v^{(3)}_{k\pm1}  + v_L u^{(3)}_{k\pm1}$\\
\hline
    $\tilde{z}_{2,k}^{(5)}$ 
    &
    $ u_L v^{(4)}_{k\pm1}  + v_L u^{(4)}_{k\pm1}$\\
\hline
    $\tilde{z}_{2,k}^{(6)}$ 
    &
    $- u_L \big( \beta_s v^{(3)}_{k\pm1} +v^{(1)}_{k\pm1} \big) - v_L \big( \beta_s u^{(3)}_{k\pm1} +u^{(1)}_{k\pm1} \big)$\\
\hline
\multicolumn{2}{|c|}{Coefficients in front of $r^3 $, for $k \geq 1$} \\
\hline
    $\tilde{z}_{3,k}^{(3)}$ 
    & 
    $u_L (u^{(1)}\ast v^{(2)})_{k\pm1} +u_L (u^{(2)}\ast v^{(1)})_{k\pm1} +v_L (u^{(1)}\ast u^{(2)})_{k\pm1}$\\
\hline
    $\tilde{z}_{3,k}^{(4)}$ & \\
    $\tilde{z}_{3,k}^{(5)}$ & $0$\\
    $\tilde{z}_{3,k}^{(6)}$ & \\
\hline
\end{tabular}
\caption{Coefficients $\tilde{z}_{2,k}^{(i)}$ and $\tilde{z}_{3,k}^{(i)}$ for $k \neq 0$.}\label{table:Coeff}
\end{table}
Thus, we set $\tilde{z}_1^{(i)}=\lbrace \tilde{z}_{1,k}^{(i)} \rbrace_{k \geq 0}$, $\tilde{z}_2^{(i)}=\lbrace \tilde{z}_{2,k}^{(i)} \rbrace_{k \geq 0}$ and $\tilde{z}_3^{(i)}=\lbrace \tilde{z}_{3,k}^{(i)} \rbrace_{k \geq 0}$. 
We note that values of $\tilde{z}_{1,0}^{(i)}$ and $\tilde{z}_{2,0}^{(i)}$  are not explicitly given, but \eqref{Boundz0} provides bounds on these terms.
We are going to abuse notation by referring to these bounds as $\tilde{z}_{1,0}^{(i)}$ and $\tilde{z}_{2,0}^{(i)}$, where we will correct for this abuse below whenever these terms get involved.
We set $\tilde{z}_{3,0}^{(i)}=0$. 

For $l=1,2$, one can estimate, using Equation~\eqref{Boundz} and the definition of $\tilde{z}^{(i)}_j$,
\begin{alignat}{1}
	 \left| \left( A[DF(\beta_s,\bx_s +b)c - A^\dagger c] \right)^{(l)} \right| \leq \hspace*{-3cm} \nonumber \\
&	\displaystyle\left(\sum_{i=1}^2 \frac{\left|A_{l,i}\right|}{\nu^m} 
	  +\sum_{i=3}^6 \left|A_{l,i}\tilde{z}_1^{(i)} \right|\right)r 	   +\sum_{i=3}^6\|A_{l,i}\|_{\infty,\nu^{-1}} \left(\|\tilde{z}_2^{(i)}\|_{1,\nu} r^2  +\|\tilde{z}_3^{(i)}\|_{1,\nu}r^3 \right), \label{eq:splittingA12}
\end{alignat}
and for $l=3,4,5,6$
\begin{alignat}{1}
    \left\|\left( A[DF(\beta_s,\bx_s +b)c -A^\dagger c] \right)^{(l)} \right\|_{1,\nu} \leq \hspace*{-4cm} \nonumber \\ 
	&\displaystyle \left(\sum_{i=1}^2  \frac{\left\|A_{l,i} \right\|_{1,\nu}}{\nu^m} +\sum_{i=3}^6 \left\|A_{l,i}\tilde{z}_1^{(i)} \right\|_{1,\nu}\right)r
	+\sum_{i=3}^6 \left\|A_{l,i}\right\|_{B(\ellnu)}\left( \|\tilde{z}_2^{(i)}\|_{1,\nu} r^2 +\|\tilde{z}_3^{(i)}\|_{1,\nu} r^3\right). \label{eq:splittingA36}
\end{alignat}
Apart from $|A_{l,i}|$ for $l,i=1,2$, which are scalars, it is not immediately obvious how to compute or estimate the terms in~\eqref{eq:splittingA12} and~\eqref{eq:splittingA36} explicitly.
%
The norms $\|A_{l,i}\|_{\infty,\nu^{-1}}$ for $l=1,2$, $i=3,4,5,6$ and 
$\|A_{l,i}\|_{1,\nu}$ for $i=1,2$, $l=3,4,5,6$
can be computed directly, since they are represented by row and column vectors of length $m$.
The operator norms $\left\|A_{l,i}\right\|_{B(\ellnu)}$ 
can be computed using Lemma~\ref{l:Blnu1norm},
since for $l \neq i$ they are represented by finite matrices, whereas for $l=i$
they have a decaying diagonal tail (see the analogous Lemma~\ref{lem:operator_norm}).

The norms $\|\tilde{z}^{(i)}_2\|_{1,\nu}$ and $\|\tilde{z}^{(i)}_3\|_{1,\nu}$ in the quadratic and cubic terms in $r$ can be estimated using the Banach algebra structure. Taking into account the bound on $\tilde{z}^{(i)}_{2,0}$
in~\eqref{Boundz0}, this leads to
bounds
\[
    \|\tilde{z}_{2}^{(i)}\|_{1,\nu} \leq \mathcal{W}^{(i)}_2 
	\qquad\text{for } i=3,4,5,6,
\]
with
\begin{alignat}{1}
 \mathcal{W}^{(3)}_2 & \bydef \tilde{\Lambda}^{(1)} +2\left( \nu + \frac{1}{\nu}\right)
	\left(
    \|\bx_0^{(1)}\|_{1,\nu} +\|\bx^{(2)}_0\|_{1,\nu} + \|\Delta\bx^{(1)}\|_{1,\nu}
      + \|\Delta\bx^{(2)}\|_{1,\nu} +1 +\overL_0 + |\Delta\overL| 
	  \right), \label{Normz32}\\
 \mathcal{W}^{(4)}_2 & \bydef \tilde{\Lambda}^{(2)} +2\left( \nu + \frac{1}{\nu}\right), \label{Normz42}\\
  \mathcal{W}^{(5)}_2 & \bydef \tilde{\Lambda}^{(3)} +2\left( \nu + \frac{1}{\nu}\right),\label{Normz52} \\
 \mathcal{W}^{(6)}_2 & \bydef \tilde{\Lambda}^{(4)} +2\left( \nu + \frac{1}{\nu} \right)\left( \beta_1 + 1 \right) , \label{Normz62}
\end{alignat}
and
\begin{align}
    \|\tilde{z}_{3}^{(3)}\|_{1,\nu} \leq \mathcal{W}^{(3)}_3 & \bydef 3\left( \nu + \frac{1}{\nu}\right) \label{Normz33}.
\end{align}
The factor $\nu+\nu^{-1}$ in the expressions above is due to the shift in index (to the right and to the left) in $u^{(i)}_{k\pm1}$, $v^{(i)}_{k\pm1}$, etc.

This leaves us with estimating $|A_{l,i}\tilde{z}_1^{(i)} |$ and $\|A_{l,i}\tilde{z}_1^{(i)} \|_{1,\nu}$. Since these appear in the terms that are linear in $r$, a direct triangle inequality bound would be too rough for the method to succeed. Hence we estimate these terms more carefully below.

For the term in front of $r$ in equation \eqref{eq:splittingA12}, for $l=1,2$, we have 
\begin{align*}
    \displaystyle \sum_{i=3}^6 \left| A_{l,i}\tilde{z}_i^{(1)} \right| 
    &\leq \sum_{i=3}^6 \left| (A_{l,i})_0 \right|  \mathcal{W}^{(i)}_1 + \sum_{i=3}^6\left|\sum_{k=1}^{m-1} (A_{l,i})_k \tilde{z}_{1,k}^{(i)} \right|. 
\end{align*}
Here we have corrected for our abuse of notation regarding $\tilde{z}^{(i)}_{1,0}$ by splitting it off using the triangle inequality.

\begin{remark}\label{r:Q}
We use the bound~\eqref{e:dualbound}
to estimate the convolution
\[
  \sup_{\|v\|_{1,\nu} \leq 1} | (a \ast v)_k | =
   \sup_{\|v\|_{1,\nu} \leq 1} \left| \sum_{k' \in \mathbb{Z}} v_{|k'|} a_{|k-k'|}   \right| \leq 
   \max \left\{ |a_k| ,\sup_{k'\geq1} \frac{|a_{|k-k'|}|+|a_{|k+k'|}|}{2 \nu^{k'}} \right\} \bydef \Q_k(a).
\]
A similar estimate leads to
\[
  \sup_{\|v\|_{1,\nu} \leq 1} | (a \ast \hv)_k |  \leq 
  \sup_{k'\geq m} \frac{|a_{|k-k'|}|+|a_{|k+k'|}|}{2 \nu^{k'}}  \bydef \hat{\Q}_k(a).
\]	
\end{remark}

Some of the terms in $\tilde{z}^{(i)}_{1,k}$ are computable directly, while others need to be estimated. To present these estimates in a structured way we
introduce several computable constants.
For the convolution terms involving either $v$ or $\hv$ in $\tilde{z}^{(3)}_{1,k}$ we introduce (for $k\geq 1$)
\begin{alignat*}{1}
  \omega^{(i)}_{k} & \bydef \Q_{k-1}(\bx^{(i)}) + \Q_{k+1}(\bx^{(i)}) , \\
  \hat{\omega}^{(i)}_{k} & \bydef \hat{\Q}_{k-1}(\bx^{(i)}) + \hat{Q}_{k+1}(\bx^{(i)}) , \\
  \Delta \omega^{(i)}_{k} & \bydef \Q_{k-1}(\Delta\bx^{(i)}) + \Q_{k+1}(\Delta\bx^{(i)}) .
\end{alignat*}
Here $\Q_k(\cdot)$ and $\hat{\Q}_k(\cdot)$, defined in Remark~\ref{r:Q}, can be computed (at least finitely many of them) since $\bx^{(i)}$ and $\Delta \bx^{(i)}$ have only finitely many nonzero components.
We now set, for $k=1,\dots,m-1$,
\begin{alignat*}{1}
 z_k^{(3)} &\bydef |\Delta\overL| \Bigl[ \frac{2}{\nu^{k-1}} +\omega^{(1)}_k +\omega^{(2)}_k \Bigr] +\big(\overL_0+|\Delta\overL|\big) \big[\Delta\omega^{(1)}_k +\Delta\omega^{(2)}_k\big] +\overL_0\big[\hat{\omega}^{(1)}_k +\hat{\omega}^{(2)}_k\big], \\
 z_k^{(4)} &\bydef 2 \frac{|\Delta\overL|}{\nu^{k-1}}, \\
 z_k^{(5)} &\bydef 2\frac{|\Delta\overL|}{\nu^{k-1}}, \\
 z_k^{(6)} & \bydef 2\frac{|\overL_0\Delta\beta +\Delta\overL\beta_0| +|\Delta\overL|}{\nu^{k-1}} +2\frac{\left|\Delta\overL\Delta\beta\right|}{\nu^{k-1}} ,
\end{alignat*}
as well as
\begin{alignat*}{1}
 \hat{z}_k^{(3)} &\bydef 
 \Delta \bx^{(2)}_{k\pm1} +(\bx_0^{(1)}\ast\Delta \bx^{(2)})_{k\pm1} +(\Delta \bx^{(1)}\ast \bx_0^{(2)})_{k\pm1},\\ 
 \hat{z}_k^{(4)} &\bydef  \Delta\bx^{(3)}_{k\pm1}, \\
 \hat{z}_k^{(5)} &\bydef  \Delta\bx^{(4)}_{k\pm1}, \\
 \hat{z}_k^{(6)} &\bydef
 \Delta\beta \bx_{0,k\pm1}^{(3)} +\Delta \bx^{(1)}_{k\pm1} +\beta_0 \Delta\bx^{(3)}_{k\pm1} ,
\end{alignat*}
and
\begin{alignat*}{1}
 \hat{\hat{z}}_k^{(3)} &\bydef ( \Delta \bx^{(1)}\ast \Delta \bx^{(2)})_{k\pm1} , \\
  \hat{\hat{z}}_k^{(6)} &\bydef \Delta\beta \Delta\bx^{(3)}_{k\pm1} .
\end{alignat*}
Recall that $|A|$ denotes the component-wise absolute value. 
Then we have the computable estimates ($l=1,2$)
\begin{alignat*}{1}
  \left|\sum_{k=1}^{m-1} (A_{l,3})_k \tilde{z}_{1,k}^{(3)} \right| 
  \leq \mathcal{Z}_{l,3}    & \bydef
 \frac{(|A|_{l,3})_{m-1}\overL_0}{\nu^{m}}  
   +\sum_{k=1}^{m-1}(|A|_{l,3})_{k}z^{(3)}_k 
   + \left|\sum_{k=1}^{m-1} (A_{l,3})_k   \hat{z}_k^{(3)} \right|   
   +\left|\sum_{k=1}^{m-1} (A_{l,3})_k \hat{\hat{z}}_k^{(3)} \right|,\\
  \left|\sum_{k=1}^{m-1} (A_{l,4})_k \tilde{z}_{1,k}^{(4)} \right| 
  	   \leq  \mathcal{Z}_{l,4} &\bydef 
	 \frac{(|A|_{l,4})_{m-1} \overL_0}{\nu^{m}} +
	 \sum_{k=1}^{m-1} (|A|_{l,4})_k z^{(4)}_k +\left|\sum_{k=1}^{m-1} (A_{l,4})_k  \hat{z}_k^{(4)} \right| , \\
  \left|\sum_{k=1}^{m-1} (A_{l,5})_k \tilde{z}_{1,k}^{(5)} \right|
 \leq \mathcal{Z}_{l,5} & \bydef 
		   \frac{(|A|_{l,5})_{m-1}\overL_0}{\nu^{m}} 
		   +\sum_{k=1}^{m-1} (|A|_{l,5})_k z^{(5)}_k +\left|\sum_{k=1}^{m-1} (A_{l,5})_k   \hat{z}_k^{(5)} \right|  , \\
  \left|\sum_{k=1}^{m-1} (A_{l,6})_k \tilde{z}_{1,k}^{(6)} \right|
  \leq \mathcal{Z}_{l,6} &\bydef
 \frac{(|A|_{l,6})_{m-1} \overL_0 (\beta_0+1)}{\nu^{m}} +\sum_{k=1}^{m-1} (|A|_{l,6})_k z^{(6)}_k 
+\left|\sum_{k=1}^{m-1} (A_{l,6})_k \hat{z}_k^{(6)} \right|
    +\left|\sum_{k=1}^{m-1} (A_{l,6})_k   \hat{\hat{z}}_k^{(6)} \right| .
\end{alignat*}
For $l=3,4,5,6$, we split 
the estimate in three terms because of the way the $\tilde{z}_{i,0}^{(1)}$ bounds and $A$ are defined. Using \eqref{Boundz0}, we get ($i,l=3,4,5,6$)
\begin{alignat}{1}
    \left\| A_{l,i}\tilde{z}_1^{(i)} \right\|_{1,\nu} & \leq 
	 \mathcal{W}^{(i)}_1 \sum_{j=0}^{m-1}\bigl|(A_{l,i})_{j0} \bigr|    
   +2\sum_{j=1}^{m-1} \left| \sum_{k=1}^{m-1} (A_{l,i})_{j,k}\tilde{z}_{1,k}^{(i)} \right|\nu^j +2 \delta_{l,i} \sum_{j \geq m} \frac{1}{2j}|\tilde{z}_{1,j}^{(i)}|\nu^j.
	\label{e:finalones}
\end{alignat}
Again we have dealt with the $\tilde{z}^{(i)}_{1,0}$ terms separately to take into account our abuse of notation. 

The final two terms in~\eqref{e:finalones} still need to be estimated.
The first of these can be estimated in the same way as above, which we write
(for $3 \leq l,i \leq 6$) compactly as
\[
\sum_{j=1}^{m-1} \left| \sum_{k=1}^{m-1} (A_{l,i})_{j,k}\tilde{z}_{1,k}^{(3)} \right|\nu^j \leq  \mathcal{Z}_{l,i} \bydef \sum_{j=1}^{m-1} 
(\mathcal{Z}_{l,i})_j,
\]
with
\begin{alignat*}{1}
	(\mathcal{Z}_{l,i})_j \bydef &
	 \frac{(|A|_{l,i})_{j,m-1} \overL_0 (\delta_{i,6}\beta_0+1)}{\nu^{m}} +\sum_{k=1}^{m-1} (|A|_{l,i})_{j,k} z^{(i)}_k 
	+\left|\sum_{k=1}^{m-1} (A_{l,i})_{j,k} \hat{z}_k^{(i)} \right|
	    +\left|\sum_{k=1}^{m-1} (A_{l,i})_{j,k}  \hat{\hat{z}}_k^{(i)} \right| ,
\end{alignat*}
where one should read $\hat{\hat{z}}^{(4)}_k=\hat{\hat{z}}^{(5)}_k=0$.

For the final `tail' terms in \eqref{e:finalones}, we bound these as we did for $z^{(i)}_2$ and $z^{(i)}_3$ coefficients. We obtain
%
%
%
\begin{alignat*}{1}
  \sum_{j \geq m} \frac{|\tilde{z}_{1,j}^{(3)}|}{j}\nu^j 
  \leq  \mathcal{Z}_3^\infty & \bydef
 \frac{1}{2m}\left( \nu + \frac{1}{\nu} \right) \left(\overL_0+|\Delta\overL|\right)\left(\|\bx^{(1)}_0\|_{1,\nu} +\| \bx^{(2)}_0\|_{1,\nu} +\|\Delta \bx^{(1)}\|_{1,\nu} +\| \Delta \bx^{(2)}\|_{1,\nu} +1 \right)\\
  &\qquad\quad+\sum_{k=m}^{2m-1} \frac{\nu^k}{k} \left(\left|(\bx^{(1)}_0 \ast \bx^{(2)}_0)_{k\pm1}\right| +\left|(\bx^{(1)}_0\ast\Delta\bx^{(2)})_{k\pm1}\right| +\left|(\Delta\bx^{(1)}\ast\bx^{(2)}_0)_{k\pm1}\right|\right) \\
  & \qquad\qquad\quad+\sum_{k=m}^{2m-1} \frac{\nu^k}{k}\left|(\Delta\bx^{(1)}\ast\Delta\bx^{(2)})_{k\pm1}\right| + \frac{\nu^m}{m}\left( |\Delta\bx^{(2)}_{m-1}| +|\bx^{(2)}_{0,m-1}| \right)   \\
  \sum_{j \geq m} \frac{|\tilde{z}_{1,j}^{(4)}|}{j}\nu^j 
  \leq  \mathcal{Z}_4^\infty & \bydef \frac{1}{2m}\left( \nu + \frac{1}{\nu} \right)\left( \overL_0 +|\Delta\overL|\right) +\frac{\nu^m}{m}\left(|\bx^{(3)}_{0,m-1}| +|\Delta\bx^{(3)}_{m-1}|\right) , \\
  \sum_{j \geq m} \frac{|\tilde{z}_{1,j}^{(5)}|}{j}\nu^j 
  \leq \mathcal{Z}_5^\infty & \bydef \frac{1}{2m}\left( \nu + \frac{1}{\nu} \right)\left( \overL_0 +|\Delta\overL|\right) +\frac{\nu^m}{m}\left(|\bx^{(4)}_{0,m-1}| +|\Delta\bx^{(4)}_{m-1}|\right) ,\\
  \sum_{j \geq m} \frac{|\tilde{z}_{1,j}^{(6)}|}{j}\nu^j 
   \leq  \mathcal{Z}_6^\infty & \bydef 
   \frac{1}{2m}\left( \nu + \frac{1}{\nu} \right) \left(\overL_0 +|\Delta\overL |\right)(1 +\beta_1)  \\
  &\qquad\quad +\frac{\nu^m}{m}\left(\beta_1\left(|\bx^{(3)}_{0,m-1}| +|\Delta\bx^{(3)}_{m-1}|\right) +|\bx^{(1)}_{0,m-1}| +|\Delta\bx^{(1)}_{m-1}|\right) .
\end{alignat*}
%
%
%
Therefore, recalling~\eqref{e:W1} and \eqref{Normz32}--\eqref{Normz33}, 
for $l=1,2$, we set
\begin{alignat*}{1}
    Z_1^{(l)} &\bydef \sum_{i=1}^2 \frac{\left|A_{l,i}\right|}{\nu^m} +
	\sum_{i=3}^6 \left| (A_{l,i})_0 \right|  \mathcal{W}^{(i)}_1 
	+\sum_{i=3}^6 \mathcal{Z}_{l,i}, \\
    Z_2^{(l)} & \bydef \sum_{i=3}^6 \|A_{l,i}\|_{\infty,\nu^{-1}} \mathcal{W}^{(i)}_2, \\
    Z_3^{(l)}& \bydef \|A_{l,3} \|_{\infty,\nu^{-1}} \mathcal{W}^{(3)}_3,
\end{alignat*}
and for $l=3,4,5,6$, we set
\begin{alignat*}{1}
    Z_1^{(l)} & \bydef \sum_{i=1}^2 \frac{\left\|A_{l,i}\right\|_{1,\nu}}{\nu^m} +
\sum_{i=3}^6  \mathcal{W}^{(i)}_1 \sum_{j=1}^{m-1}\left|(A_{l,i})_{j,0}\right| 
+\sum_{i=3}^6\mathcal{Z}_{l,i} +\mathcal{Z}_l^\infty, \\
    Z_2^{(l)}& \bydef\sum_{i=3}^6 \|A_{l,i}\|_{B(\ellnu)}
	\mathcal{W}^{(i)}_2 , \\
    Z_3^{(l)}& \bydef \left\|A_{l,3}\right\|_{B(\ellnu)} 
	\mathcal{W}^{(3)}_3.
\end{alignat*}
Finally, by construction,
\begin{equation*}
  \displaystyle \sup_{b,c \in B(r)} \left\| \left(DT(\beta_s,\bx_s +b)c\right)^{(l)} \right\|_{X^{(l)}} \leq\left(Z_0^{(l)}+Z_1^{(l)} \right)r + Z_2^{(l)}r^2 +Z_3^{(l)}r^3,
  \qquad\text{for all } s \in [0,1] \text{ and } l=1,\hdots,6.
\end{equation*}


\subsection{Use of the uniform contraction principle}

Using the computable bounds $Y^{(l)}$ and $Z^{(l)}$ constructed in the previous two sections, we set
\begin{equation} \label{eq:radii_polynomials_bvp}
p^{(l)}(r) \bydef Y^{(l)} + \left(Z_0^{(l)}+Z_1^{(l)} -1 \right)r + Z_2^{(l)} r^2 +Z_3^{(l)} r^3, \qquad l=1,\dots,6.
\end{equation}
If we find an $r>0$ such that $p^{(l)}(r)<0$ for all $l=1,\ldots,6$, then according to Proposition~\ref{prop:Radii} we have validated the numerical approximation $\overx_s$ of solutions to the BVP~\eqref{e:reducedproblem}, 
for every $s \in [0,1]$, and hence we have proven the existence of symmetric homoclinic orbits for all $\beta \in [\beta_0,\beta_1]$.

\begin{proposition}
For every $s\in [0,1]$, let 
\begin{equation*}
\overline{v}_s^{(i)}(t) = \bx_{s,0}^{(i)} +2\displaystyle \sum_{k=0}^{m-1} \bx_{s,k}^{(i)} T_k(t),
\qquad\text{for } i=1,2,3,4,
\end{equation*} 
be the approximate solution of~\eqref{e:reducedproblem} that we have computed for $\beta=\beta_s$, $L=\overL_s$ and $\psi=\overp_s$.
Assume that there exists an $r>0$ such that $p^{(l)}(r)<0$ for all $l=1,\ldots,6$. Then, for each $s\in[0,1]$, there exists a 
solution of~\eqref{e:reducedproblem} 
for $\beta=\beta_s$  of the form
\begin{equation*}
v_s^{(i)}(t) = x_{s,0}^{(i)} +2\displaystyle \sum_{k=0}^{\infty} x_{s,k}^{(i)} T_k(t), \qquad\text{for } i=1,2,3,4,
\end{equation*}
and some $L=L_s$ and $\psi=\psi_s$ satisfying $|L_s-\overL_s|\leq r$ and $|\psi_s-\overp_s|\leq r$.
This solution corresponds to a (symmetric) homoclinic orbit of~\eqref{eq:brigdeODE}.
Furthermore, let
\begin{equation*}
	g^{(i)}_{s}(t) = v_s^{(i)}(t) - \overline{v}_s^{(i)}(t)
	\qquad\text{for } i=1,2,3,4,
\end{equation*}
then we have the following uniform error bound on the (central part of) the homoclinic orbit in phase space:
$ |g^{(i)}_{s}(t) | \leq r $ for all $t\in [-1,1]$, $s\in[0,1]$ and $i=1,2,3,4$.
\end{proposition}
\begin{proof}
Proposition~\ref{prop:Radii} yields that, for each $s\in[0,1]$, there exists a unique fixed point $x_s$ of $T(\beta_s,\cdot)$ in the ball of radius $r$ around $\bx_s$. 
The operator $A$ is injective since its non-diagonal part $A^{[m]}$ is invertible. The latter follows from the fact that, see~\eqref{e:Z0},
\[
  \bigl\| I_{4m+2}-A^{[m]}D_x\overline{F}(\beta_0,\bx_0) \bigr\|_{B(X^{[m]})} \leq \max_{1\leq l\leq 6} Z_0^{(l)} < 1 ,
\] 
where the final inequality is implied by $p^{(l)}(r)<0$. Here the operator norm on $X^{[m]} \cong \R^{4m+2}$ is induced by the one on $X$.
Hence the fixed point $x_s$ of $T$ solves $F(\beta_s,x_s)=0$,
and by construction $v_s$ is a solution of~\eqref{e:reducedproblem}, which through the change of variables from 
Section~\ref{s:rewriting} corresponds to a homoclinic solution of~\eqref{eq:brigdeODE}. The error bound follows from
\begin{align*}
\|v_s^{(i)}(t)-\overline{v}_s^{(i)}(t)\|_\infty &= \Bigl\| x_{s,0}^{(i)}-\bx_{s,0}^{(i)} + 2\sum_{k\geq 1}\bigl( x^{(i)}_{s,k}-\bx^{(i)}_{s,k}\bigr) T_k(t) \Bigr\|_\infty \\
&\leq \bigl|x^{(i)}_{s,0}-\bx^{(i)}_{s,0}\bigr| +2\sum_{k\geq 1} \bigl|x^{(i)}_{s,k}-\bx^{(i)}_{s,k}\bigr|\\
&\leq \bigl|x^{(i)}_{s,0}-\bx^{(i)}_{s,0}\bigr| +2\sum_{k\geq 1} \bigl|x^{(i)}_{s,k}-\bx^{(i)}_{s,k}\bigr|\nu^k \\
&= \bigl\|x_{s}^{(i)}-\bx_{s}^{(i)}\bigr\|_{1,\nu} 
\leq \| x_s - \bx_s\|_X \leq r .\qedhere
\end{align*}
\end{proof}


\section{Algorithm and results}
\label{s:algorithm}

In this section we discuss some algorithmic issues. In particular, we explain how certain computational constants are chosen and how the two parts of the problem (the manifold computation and the boundary value problem) are joined together to produce the homoclinic orbit.
To get the continuation started, the first thing to do is to compute the approximation of the manifold for a fixed value of $\beta$. Since the first coefficients of the parameterization depend on the steady state and the eigenvectors, which are known, one can start Newton's method with these values in combination with zeros for all higher order coefficients.
If Newton's method does not converge, replacing the starting point with a good approximation for a slightly higher number of Taylor coefficients (which can be computed recursively) will work. 
Once one good approximation has been found for a particular value of the parameter $\beta$, one can use it as the starting point to find another approximation for sufficiently close values of the parameter. 

Another important point for the computations is the size of the manifold that we get. Since the stable eigenvalues of the Jacobian at 0 are complex conjugates, we know that asymptotically the orbit spirals toward the origin. If the manifold we compute is large enough to contain most of the spiraling part, then we do not have to compute that part of the orbit using Chebyshev series, which is advantageous. Generally speaking, the larger the manifold is, the easier the remaining part with Chebyshev will be. Therefore we use the method developed in \cite{maximizing_manifold} to maximize the image of the parameterization we compute.

A natural approach to obtain a larger manifold is to try and maximize the $\tilde\nu$ for which we can validate the parameterization (we recall that its domain of definition is $D_{2,\tilde \nu}(\mathbb{R}^2)= \left\lbrace \theta\in\R^2,\ \left\vert\theta\right\vert_2 \leq \tilde \nu\right\rbrace$). However, taking $\tilde\nu\gg 1$ or $\tilde\nu \ll 1$ leads to numerical instabilities (see for instance the quantities $K^{(i,j)}$ defined in Section~\ref{sec:Z0}). The key observation from~\cite{maximizing_manifold} to avoid this phenomenon is the following. Given a parameterization
\begin{equation*}
P(\theta)=\sum_{\vert\alpha\vert\geq 0}a_{\alpha}\theta^{\alpha},
\end{equation*}
and, for some $\gamma>0$, a \emph{rescaled} parameterization (also with rescaled eigenvectors)
\begin{equation*}
\tilde P(\theta)=\sum_{\vert\alpha\vert\geq 0}\tilde a_{\alpha}\theta^{\alpha}, \quad \text{with}\quad \tilde a_{\alpha}=\gamma^{\vert\alpha\vert}a_{\alpha},
\end{equation*}
the parameterization $P$ on the domain $D_{2,\gamma}(\mathbb{R}^2)$ defines the same manifold as the rescaled parameterization $\tilde P$ on the domain $D_{2,1}(\mathbb{R}^2)$. Therefore we can fix $\tilde \nu$ to be $1$ and instead look for the largest $\gamma$ for which we can validate the rescaled parameterization.

Another useful feature of the results of \cite{maximizing_manifold} is that they provide the explicit dependence of the bounds $Y$ and $Z$ with respect to the rescaling $\gamma$, enabling us to recompute bounds for any rescaling cheaply. In practice, we use the following process:
\begin{itemize}
\item Compute an approximate parameterization $P$ (that is, the coefficient $a_{\alpha}$).
\item Compute the bounds $Y$ and $Z$ for $\beta_0$, without the continuation (i.e.\ take $\Delta a=0$ and $\Delta\beta=0$ in every estimate).
\item Find the largest $\gamma$ for which the proof succeeds (i.e.\ we find an $r>0$ such that $p^{(i)}(r)<0$ for all $i=1,2,3,4$, where the four radii polynomials $p^{(i)}$ are defined in \eqref{eq:radii_polynomials_manifold}) for the rescaled coefficients $\tilde a_{\alpha}=\gamma^{\vert\alpha\vert}a_{\alpha}$, while requiring the coefficients of the linear term (the one front of $r$) in each radii polynomial $p^{(i)}$ to be less than some threshold $\eta\in(0,1)$, which will be discussed below. This step yields a parameterization $\tilde P$ with rigorous error bounds on the domain $D_{2,1}(\mathbb{R}^2)$.
\item Use the parameterization $\tilde P$ with this $\gamma$ for the Chebyshev part and for the continuation.
\end{itemize}
Before describing in more detail the process of continuation, let us explain the role of the threshold $\eta$. Finding a positive root of a radii polynomial is impossible if its linear term is not negative, because all its other coefficients are always non-negative by construction. If the linear term is just negative enough for the proof to work at the single parameter value $\beta_0$, then $\Delta\beta$ has to be taken extremely small for it to remain negative for the uniform proof, since all bounds become worse monotonically in $|\Delta\beta|$. However, we want to take $\Delta\beta$ as large as possible to reduce the number of steps we have to perform to prove the existence of a symmetric homoclinic orbit for all $\beta\in [0.5,1.9]$. Hence, the addition of this threshold $\eta$ is a trade off: we get a manifold that is a bit smaller than what we could have had optimally, which makes the proof for the Chebyshev part a bit harder, but we can take larger steps in $\beta$, making the total process faster overall. In practice, we use an $\eta$ close to $0.5$ (the value we use varies slightly with $\beta$).

Once the approximation for the manifold is maximized and proven for a particular value of $\beta_0$, one can use it as the starting point to find the approximation for $\beta_1>\beta_0$ in order to compute an approximation for the whole interval $[\beta_0,\beta_1]$. We use the same rescaling $\gamma$ for the entire interval $[\beta_0,\beta_1]$. On the other hand, it is possible to use different scalings for consecutive intervals. 

The value of $\Delta\beta= \beta_1 - \beta_0$ that we use is not constant, and varies between $2.5\times 10^{-4}$ and $3.9\times 10^{-6}$. The smaller values are needed when $\beta_0\geq 1.8$. This is due to the fact that proof of the stable manifold becomes harder and harder when $\beta$ approaches $2$. Indeed, when $\beta$ goes to $2$ the real part of the stable eigenvalues (see~\eqref{eq:lambda}) goes to zero, and the problem of finding the stable manifold becomes singular (this can also be seen in the bounds derived in Section~\ref{Manifold}). 
Note that when the proof fails for a given interval, a smaller $\Delta\beta$ needs to be used. Thus, the algorithm needs to recompute both the manifold and the orbit for $\beta=\beta_1$. However, $A^\dagger$ and $A$ need not to be computed again for the new proofs since they both only depend on the approximation at $\beta=\beta_0$. 

For the manifold all proofs were done using $N=30$ for the dimension of the truncated power series. For the orbit, the proof succeeds with $m=350$ for $[\beta_0,\beta_1]\subset [0.5,1.8]$, and with $m=400$ otherwise. 
In Figure~\ref{figure:DecayAndProfile} one can see the profile of the solution for $\beta=0.5$, $\beta=1.2$ and $\beta=1.9$. The left part of the figure shows the decay rate of the solution using the logarithm of the absolute value of the first $50$ Chebyshev coefficients. Recall that the first component of the system is given by $v_1=e^{u_1}-1$, where $u_1$ is the first component of the original system, obtained after transforming the fourth order equation to a first order system. One can see that the solution for $\beta=0.5$ is really close to 
$-1$ for a much longer period of time than the other solutions depicted. This behaviour has an impact on the decay of the corresponding Chebyshev series. Moreover, another value affecting the decay rate of the solution is the time rescaling factor $L$ of the orbit. For $\beta=0.5$ (respectively $\beta=1.2$ and $\beta=1.9$) we have $L\approx 3.1312$ (respectively $L\approx 1.7671$ and $L\approx 2.6170$). The first three components of the solution and the local manifold can be seen in Figure~\ref{figure:SolutionManifold05}, Figure~\ref{figure:SolutionManifold12} and Figure~\ref{figure:SolutionManifold19} for $\beta=0.5$, $\beta=1.2$ and $\beta=1.9$, respectively.
The profiles of the first component $v^{(1)}$ of these three solutions can be compared in Figure~\ref{figure:SolProfile}, where half the symmetric homoclinic orbits is depicted. Furthermore, the three corresponding homoclinic solutions of the suspension bridge equation~\eqref{eq:ode} in the original $u$-variable are presented in Figure~\ref{figure:SolProfile_u}.

\begin{figure}
\begin{center}
\includegraphics[width=0.8\textwidth]{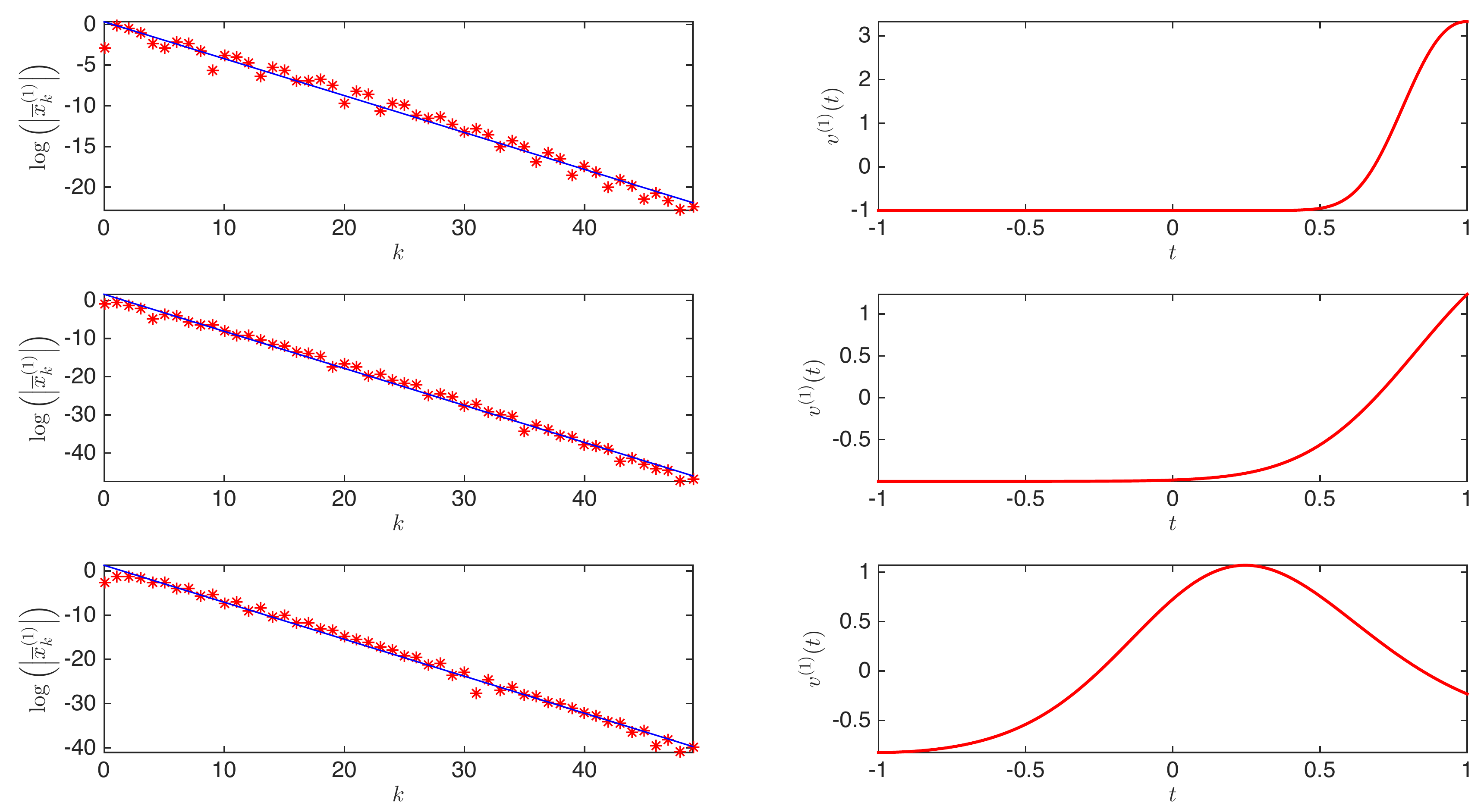}
\end{center}
\vspace{-.7cm}
\caption{The logarithm of the absolute value of the $50$ first coefficients of the first component on the left, and the profile of the first component of the solution on the right. At the top $\beta=0.5$, in the middle $\beta=1.2$ and at the bottom $\beta=1.9$}\label{figure:DecayAndProfile}
\end{figure}

Finally, to perform the proof successfully for the entire interval range $\beta \in [0.5,1.9]$ we had to execute the algorithm $7960$ times. Each proof took between $7$ and $10$ seconds on a laptop with an Intel Core i7 4500U processor on MATLAB R2016a. The code which was used to perform the proofs is available at \cite{webpage} and uses the interval arithmetic package INTLAB \cite{Ru99a}. 

\begin{figure}
\begin{center}
\includegraphics[width=0.8\textwidth]{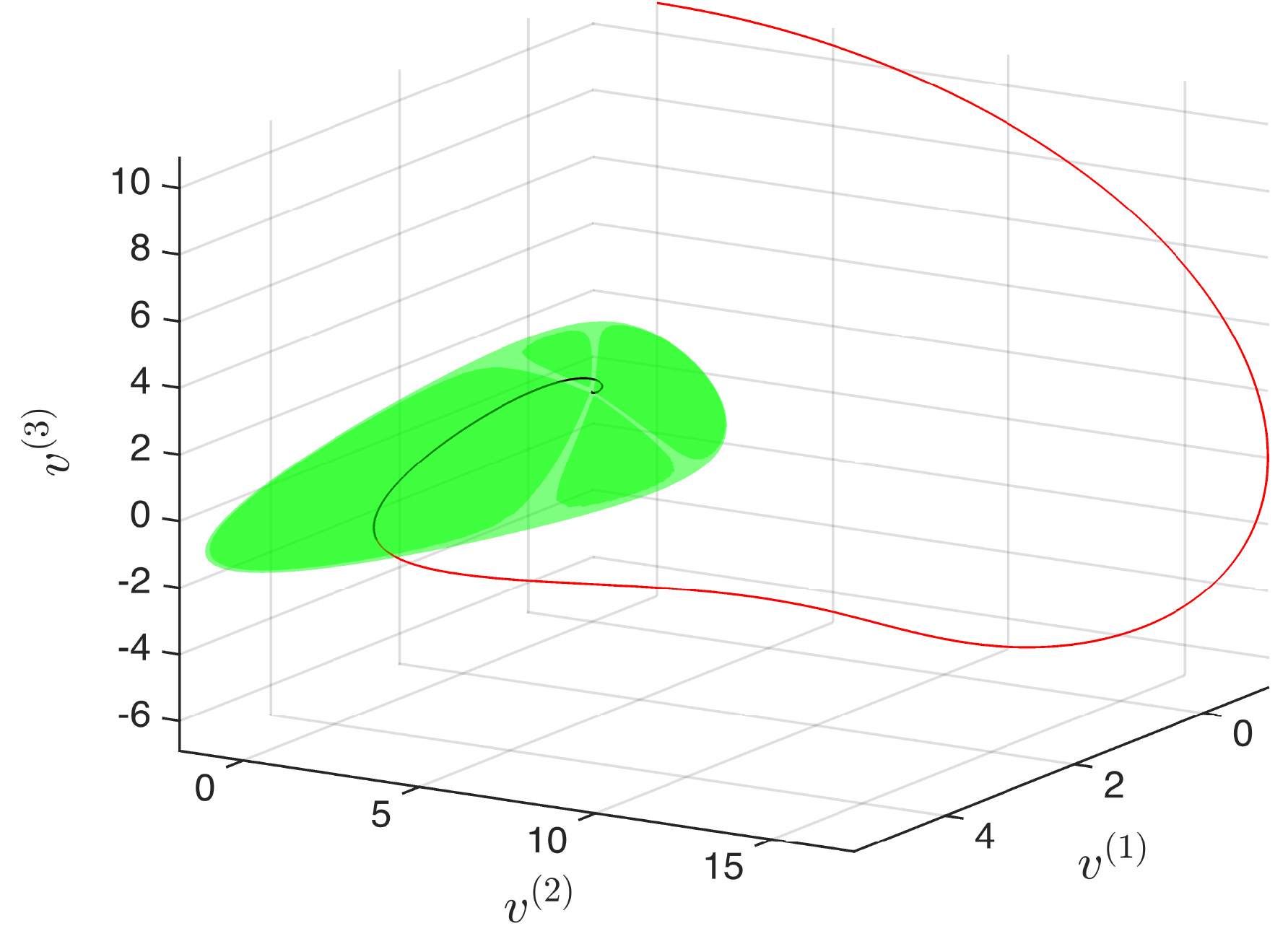}
\end{center}
\vspace{-.5cm}
\caption{First three components of the solution (red) and the manifold (green) in the case $\beta=0.5$. 
The segment in black corresponds to the forward orbit of the solution on the local manifold, where the dynamics is obtained via the 
conjugacy relation satisfied by the parameterization (e.g.\ see \cite{MR2177465}).
}\label{figure:SolutionManifold05}
\end{figure}

\begin{figure}
\begin{center}
\includegraphics[width=0.8\textwidth]{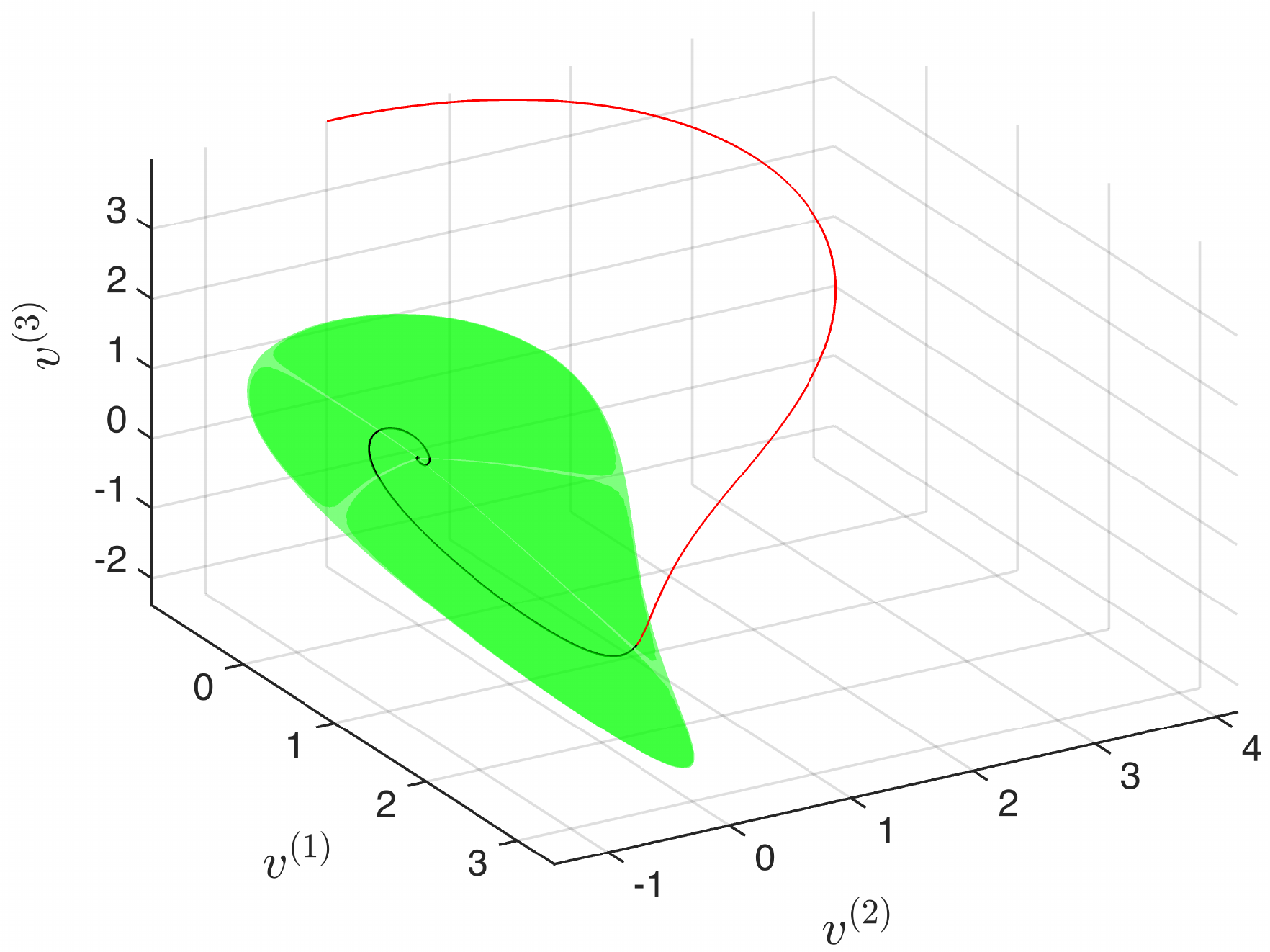}
\end{center}
\vspace{-.5cm}
\caption{First three components of the solution (red) and the manifold (green) in the case $\beta=1.2$. 
The segment in black corresponds to the forward orbit of the solution on the local manifold, where the dynamics is obtained via the 
conjugacy relation satisfied by the parameterization.
}\label{figure:SolutionManifold12}
\end{figure}

\begin{figure}
\begin{center}
\includegraphics[width=0.8\textwidth]{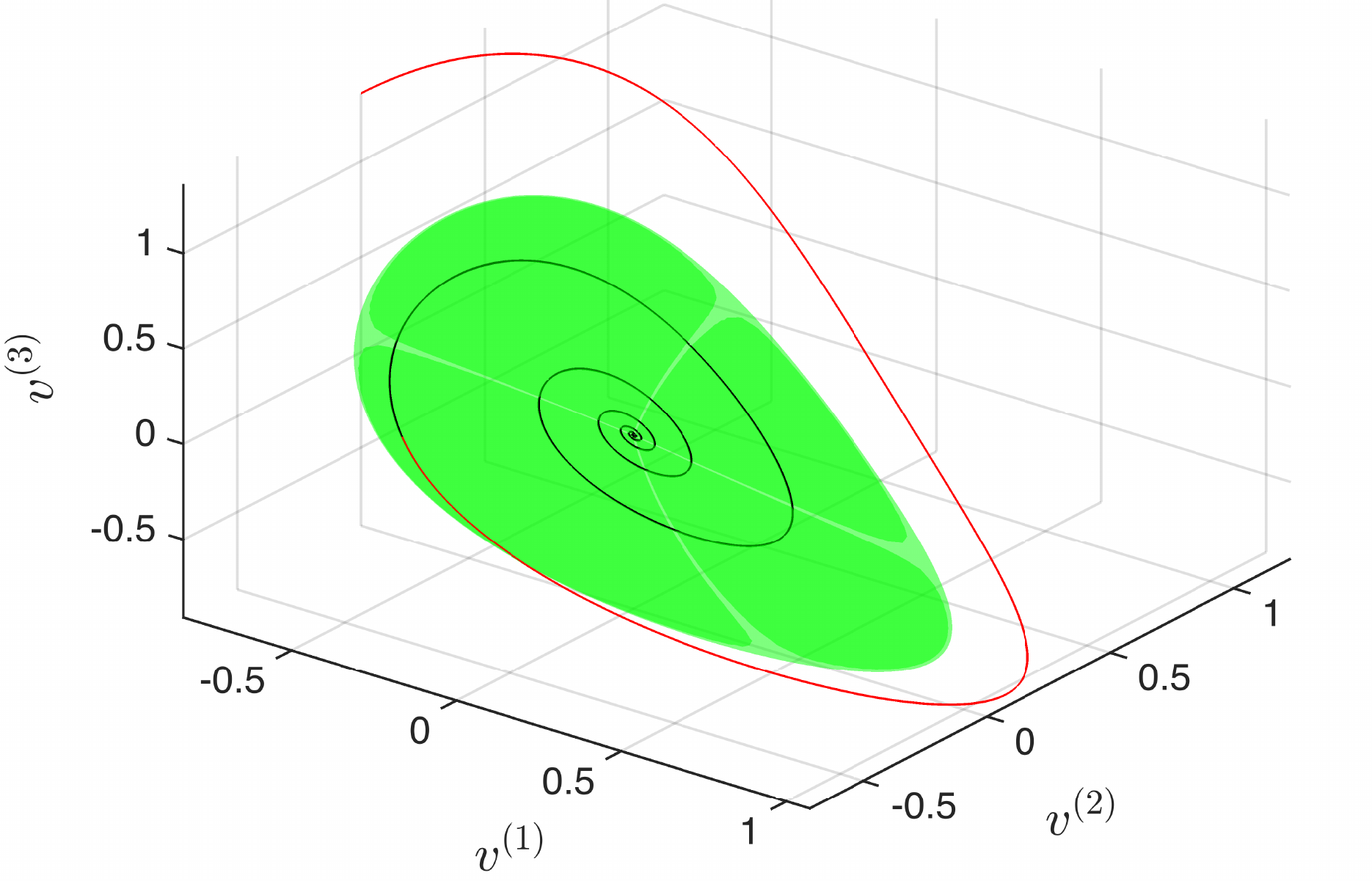}
\end{center}
\vspace{-.5cm}
\caption{First three components of the solution (red) and the manifold (green) in the case $\beta=1.9$. 
The segment in black corresponds to the forward orbit of the solution on the local manifold, where the dynamics is obtained via the 
conjugacy relation satisfied by the parameterization.
}\label{figure:SolutionManifold19}
\end{figure}

\begin{figure}
\begin{center} 
\includegraphics[width=0.8\textwidth]{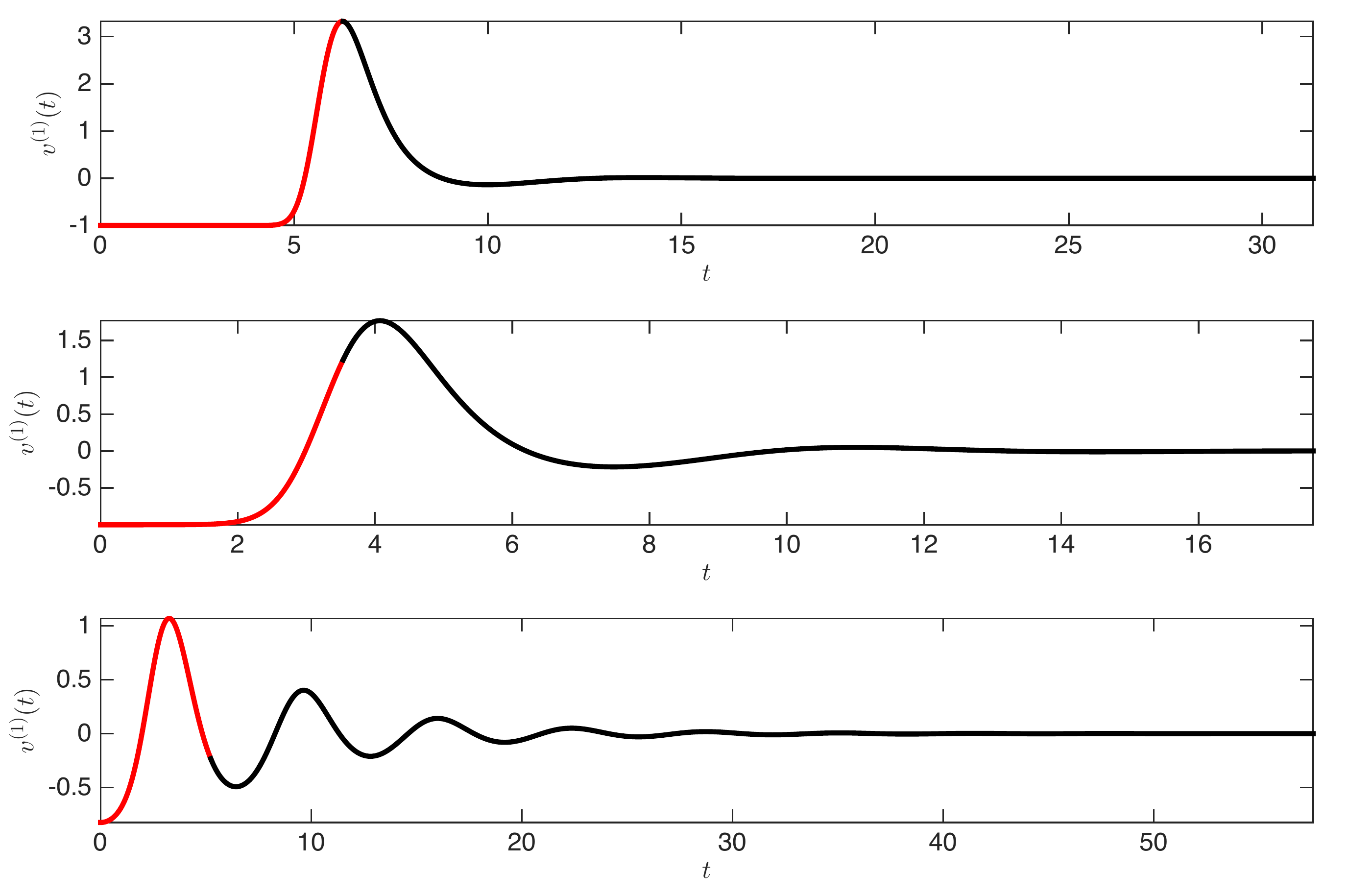}
\end{center}
\vspace{-.7cm}
\caption{The solution profiles of $v^{(1)}(t)$ for $\beta=0.5$ (top), $\beta=1.2$ (middle) and $\beta=1.9$ (bottom). The parts in red correspond to the part of the solution which was obtained using 
Chebyshev series, while the parts in black correspond to the part of the solution lying in the local stable manifold computed using Taylor series.}\label{figure:SolProfile}
\end{figure}

\begin{figure}
\begin{center}
\includegraphics[width=0.8\textwidth]{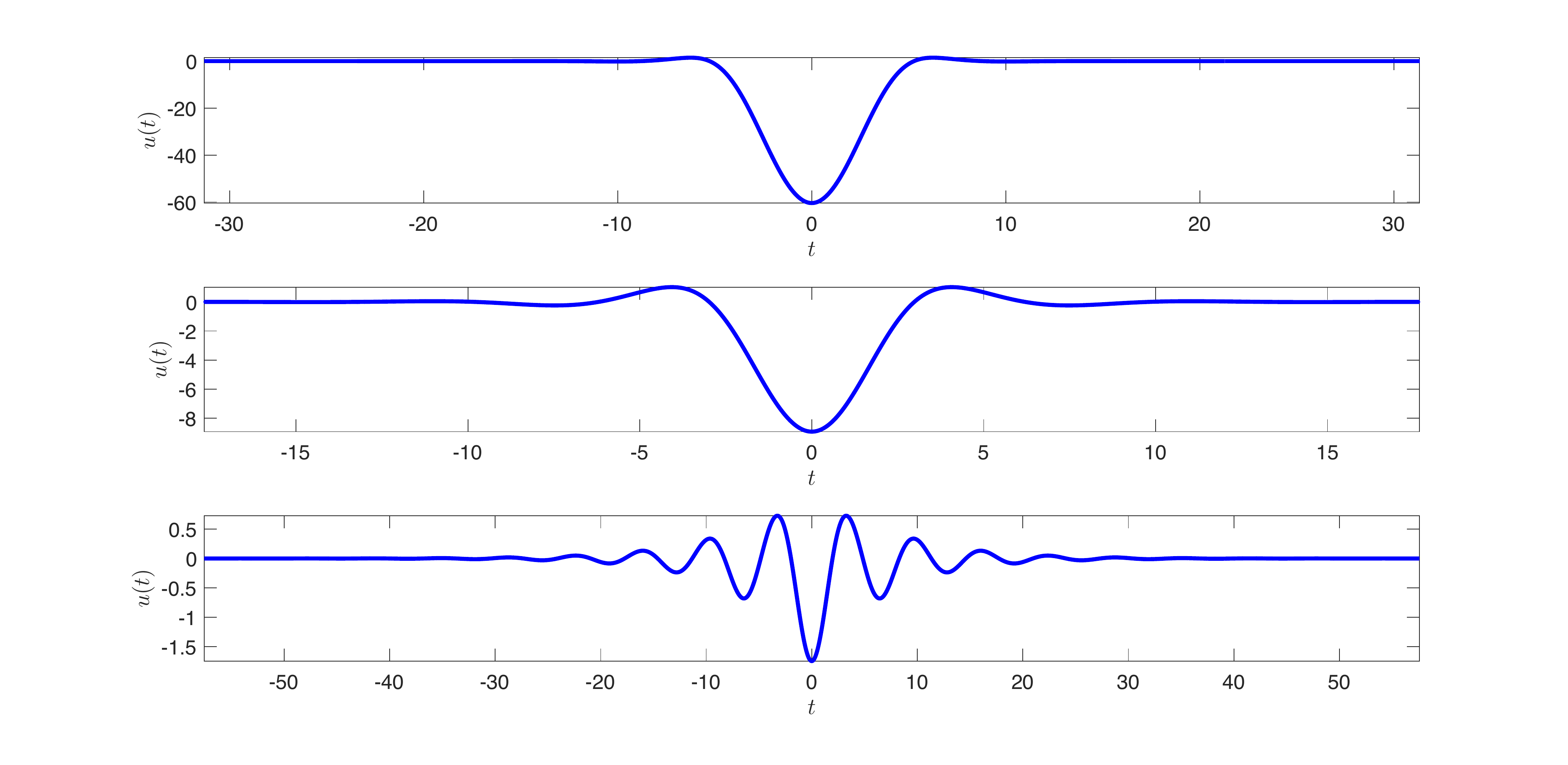}
\end{center}
\vspace{-.7cm}
\caption{The solution profiles in the variable $u$ of the suspension bridge equation~\eqref{eq:ode} for $\beta=0.5$ (top), $\beta=1.2$ (middle) and $\beta=1.9$ (bottom). Notice the different scales of the $y$-axis for the three solutions.}\label{figure:SolProfile_u}
\end{figure}


\section*{Acknowledgements}

Jan Bouwe van den Berg was supported by the grant NWO Vici-grant 016.123.606.
Jean-Philippe Lessard was supported by Gouvernement du Canada/Natural
Sciences and Engineering Research Council of Canada (NSERC), CG100747. 



\begin{thebibliography}{10}

\bibitem{maximizing_manifold}
M.~Breden, J.-P. Lessard, and J.D. Mireles~James.
\newblock Computation of maximal local (un)stable manifold patches by the
  parameterization method.
\newblock {\em Indag. Math. (N.S.)}, 27(1):340--367, 2016.

\bibitem{MR3125637}
M.~Breden, J.-P. Lessard, and M.~Vanicat.
\newblock Global bifurcation diagrams of steady states of systems of {PDE}s via
  rigorous numerics: a 3-component reaction-diffusion system.
\newblock {\em Acta Appl. Math.}, 128:113--152, 2013.

\bibitem{MR2220064}
B.~Breuer, J.~Hor{\'a}k, P.J. McKenna, and M.~Plum.
\newblock A computer-assisted existence and multiplicity proof for travelling
  waves in a nonlinearly supported beam.
\newblock {\em J. Differential Equations}, 224(1):60--97, 2006.

\bibitem{MR1388874}
B.~Buffoni, A.~R. Champneys, and J.~F. Toland.
\newblock Bifurcation and coalescence of a plethora of homoclinic orbits for a
  {H}amiltonian system.
\newblock {\em J. Dynam. Differential Equations}, 8(2):221--279, 1996.

\bibitem{parameterization_method}
X.~Cabr{\'e}, F.~Ernest, and R.~de~la Llave.
\newblock The parameterization method for invariant manifolds {I}: manifolds
  associated to non-resonant subspaces.
\newblock {\em Indiana University mathematics journal}, 52(2):283--328, 2003.

\bibitem{MR1976080}
X.~Cabr{\'e}, E.~Fontich, and R.~de~la Llave.
\newblock The parameterization method for invariant manifolds. {II}.
  {R}egularity with respect to parameters.
\newblock {\em Indiana Univ. Math. J.}, 52(2):329--360, 2003.

\bibitem{MR2177465}
X.~Cabr{\'e}, E.~Fontich, and R.~de~la Llave.
\newblock The parameterization method for invariant manifolds. {III}.
  {O}verview and applications.
\newblock {\em J. Differential Equations}, 218(2):444--515, 2005.

\bibitem{ChenMcKenna}
J.~Chen and P.J. McKenna.
\newblock Travelling waves in a nonlinearly suspended beam: theoretical results
  and numerical observations.
\newblock {\em J. Differential Equations}, 136:325--355, 1997.

\bibitem{MR660633}
S.N. Chow and J.K. Hale.
\newblock {\em Methods of bifurcation theory}, volume 251 of {\em Grundlehren
  der Mathematischen Wissenschaften [Fundamental Principles of Mathematical
  Science]}.
\newblock Springer-Verlag, New York, 1982.

\bibitem{CL}
A.~Correc and J.-P. Lessard.
\newblock Coexistence of nontrivial solutions of the one-dimensional
  ginzburg-landau equation: a computer-assisted proof.
\newblock {\em European Journal of Applied Mathematics}, 26(1):33--60, 2015.

\bibitem{MR2338393}
S.~Day, J.-P. Lessard, and K.~Mischaikow.
\newblock Validated continuation for equilibria of {PDE}s.
\newblock {\em SIAM J. Numer. Anal.}, 45(4):1398--1424, 2007.

\bibitem{MR3119065}
F.~Gazzola.
\newblock Nonlinearity in oscillating bridges.
\newblock {\em Electron. J. Differential Equations}, pages No. 211, 47, 2013.

\bibitem{MR3467671}
A.~Haro, M.~Canadell, J.-L. Figueras, A.~Luque, and J.-M. Mondelo.
\newblock {\em The parameterization method for invariant manifolds}, volume 195
  of {\em Applied Mathematical Sciences}.
\newblock Springer, [Cham], 2016.
\newblock From rigorous results to effective computations.

\bibitem{LR}
J.-P. Lessard and C.~Reinhardt.
\newblock Rigorous numerics for nonlinear differential equations using
  {C}hebyshev series.
\newblock {\em SIAM J. Numer. Anal.}, 52(1):1--22, 2014.

\bibitem{McKennaWalter}
P.J. McKenna and W.~Walter.
\newblock Travelling waves in a suspension bridge.
\newblock {\em SIAM J. Appl. Math.}, 50:703--715, 1990.

\bibitem{MirelessMischaikow}
J.D. Mireles-James and K.~Mischaikow.
\newblock Rigorous a posteriori computation of (un)stable manifolds and
  connecting orbits for analytic maps.
\newblock {\em SIAM J. Appl. Dyn. Syst.}, 2:957--1006, 2013.

\bibitem{MR1632819}
L.A. Peletier and W.C. Troy.
\newblock Multibump periodic travelling waves in suspension bridges.
\newblock {\em Proc. Roy. Soc. Edinburgh Sect. A}, 128(3):631--659, 1998.

\bibitem{MR1839555}
L.A. Peletier and W.C. Troy.
\newblock {\em Spatial patterns}.
\newblock Progress in Nonlinear Differential Equations and their Applications,
  45. Birkh\"auser Boston, Inc., Boston, MA, 2001.
\newblock Higher order models in physics and mechanics.

\bibitem{Ru99a}
{S.M.} Rump.
\newblock {INTLAB - INTerval LABoratory}.
\newblock In Tibor Csendes, editor, {\em {Developments~in~Reliable Computing}},
  pages 77--104. Kluwer Academic Publishers, Dordrecht, 1999.
\newblock http://www.ti3.tu-harburg.de/rump/.

\bibitem{MR2578798}
S.~Santra and J.~Wei.
\newblock Homoclinic solutions for fourth order traveling wave equations.
\newblock {\em SIAM J. Math. Anal.}, 41(5):2038--2056, 2009.

\bibitem{MR1929147}
D.~Smets and J.B. van~den Berg.
\newblock Homoclinic solutions for {S}wift-{H}ohenberg and suspension bridge
  type equations.
\newblock {\em J. Differential Equations}, 184(1):78--96, 2002.

\bibitem{SzczelinaZgli}
R.~Szczelina and P.~Zgliczy\'{n}ski.
\newblock A homoclinic orbit in a planar singular ode-a computer assisted
  proof.
\newblock {\em SIAM J. Appl. Dyn. Syst.}, 12(3):1541--1565, 2013.

\bibitem{webpage}
J.B. van~den Berg, M.~Breden, J.-P. Lessard, and M.~Murray.
\newblock {MATLAB} code for ``{C}ontinuation of homoclinic orbits in the
  suspension bridge equation: a computer-assisted proof'', 2017.
\newblock \verb+http://www.math.vu.nl/~janbouwe/code/suspensionbridge/+.

\bibitem{BDLM}
J.B. van~den Berg, A.~Desch\^enes, J.-P. Lessard, and J.D. Mireles~James.
\newblock Stationary coexistence of hexagons and rolls via rigorous
  computations.
\newblock {\em SIAM J. Appl. Dyn. Syst.}, 14(2):942--979, 2015.

\bibitem{MR2630003}
J.B. van~den Berg, J.-P. Lessard, and K.~Mischaikow.
\newblock Global smooth solution curves using rigorous branch following.
\newblock {\em Math. Comp.}, 79(271):1565--1584, 2010.

\bibitem{BLMM}
J.B. van~den Berg, J.D. Mireles-James, J.-P. Lessard, and K.~Mischaikow.
\newblock Rigorous numerics for symmetric connecting orbits: even homoclinics
  of the {G}ray-{S}cott equation.
\newblock {\em SIAM J. Math. Anal.}, 43(4):1557--1594, 2011.

\bibitem{BMR}
J.B. van~den Berg, J.D. Mireles-James, and C.~Reinhardt.
\newblock Computing (un)stable manifolds with validated error bounds:
  non-resonant and resonant spectra.
\newblock {\em J. Nonlinear Sci.}, 26(4):1055--1095, 2016.

\bibitem{RayJB}
J.B. van~den Berg and R.S.S. Sheombarsing.
\newblock Rigorous numerics for {ODEs} using {C}hebyshev series and domain
  decomposition, 2016.
\newblock Preprint.

\bibitem{Wilczak}
D.~Wilczak.
\newblock Symmetric heteroclinic connections in the michelson system: a
  computer assisted proof.
\newblock {\em SIAM J. Appl. Dyn. Syst.}, 4(3)(electronic):489--514, 2005.

\bibitem{Wilczak2}
D.~Wilczak.
\newblock The existence of shilnikov homoclinic orbits in the michelson system:
  a computer assisted proof.
\newblock {\em Found. Comput. Math.}, 6(4):495--535, 2006.

\bibitem{WilczakZgli}
D.~Wilczak and P.~Zgliczy\'{n}ski.
\newblock Heteroclinic connections between periodic orbits in planar restricted
  circular three-body problem - a computer assisted proof.
\newblock {\em Comm. Math. Phys.}, 234(1):37--75, 2003.

\bibitem{WojcikZgli}
K.~W\'{o}jcik and P.~Zgliczy\'{n}ski.
\newblock On existence of infinitely many homoclinic solutions.
\newblock {\em Monatsh. Math.}, 130(2):155--160, 2000.

\end{thebibliography}
\end{document}